\documentclass[a4paper,10pt,reqno]{smfart}
\usepackage[T1]{fontenc}
\usepackage[utf8]{inputenc}
\usepackage{lmodern}
\usepackage{amsmath} 
\usepackage{amsfonts}
\usepackage{amscd} 
\usepackage{amssymb} 
\usepackage{enumerate}
\usepackage{mathrsfs}
\usepackage[francais]{babel}
\usepackage{url}
\usepackage{import}
\usepackage[all]{xy}
\def \beq {\begin{equation}}
\def \eeq {\end{equation}}
\newtheorem{thm}{Th\'eor\`eme}[section]
\newtheorem{prop}[thm]{Proposition}
\newtheorem{lemme}[thm]{Lemme}
\newtheorem{cor}[thm]{Corollaire}
\newtheorem{question}[thm]{Question}
\newtheorem{hyp}[thm]{Hypoth\`ese}

\newtheorem{hyps}[thm]{Hypoth\`eses}

\theoremstyle{definition}
\newtheorem{defi}[thm]{D\'efinition}

\newtheorem{nota}[thm]{Notation}
\theoremstyle{remark}

\newtheorem{rem}[thm]{Remarque}

\makeatletter
\@addtoreset{equation}{subsection}
\renewcommand{\theequation}{\arabic{section}.\arabic{subsection}.\arabic{equation}}
\makeatother
\newcommand{\termin}[1]{{\em #1}}
\newcommand{\bA}{\mathbf{A}}
\newcommand{\bC}{\mathbf{C}}
\newcommand{\bG}{\mathbf{G}}
\newcommand{\bN}{\mathbf{N}}
\newcommand{\bP}{\mathbf{P}}
\newcommand{\bR}{\mathbf{R}}
\newcommand{\bZ}{\mathbf{Z}}
\newcommand{\bd}{\boldsymbol{d}}
\newcommand{\be}{\boldsymbol{e}}
\newcommand{\bbf}{\boldsymbol{f}}
\newcommand{\bg}{\boldsymbol{g}}
\newcommand{\bs}{\boldsymbol{s}}
\newcommand{\bt}{\boldsymbol{t}}
\newcommand{\bx}{\boldsymbol{x}}
\newcommand{\by}{\boldsymbol{y}}

\newcommand{\balpha}{\boldsymbol{\alpha}}
\newcommand{\bbeta}{\boldsymbol{\beta}}
\newcommand{\bgamma}{\boldsymbol{\gamma}}
\newcommand{\bnu}{\boldsymbol{\nu}}
\newcommand{\bmu}{\boldsymbol{\mu}}
\newcommand{\cC}{{\mathcal C}}

\newcommand{\cI}{{\mathcal I}}
\newcommand{\cJ}{{\mathcal J}}
\newcommand{\cK}{{\mathcal K}}
\DeclareMathAlphabet{\eulercal}{U}{eus}{m}{n}
\newcommand{\ecC}{{\eulercal C}}
\newcommand{\ecD}{{\eulercal D}}
\newcommand{\ecE}{{\eulercal E}}
\newcommand{\ecF}{{\eulercal F}}
\newcommand{\ecG}{{\eulercal G}}
\newcommand{\ecH}{{\eulercal H}}

\newcommand{\ecN}{{\eulercal N}}
\newcommand{\ecO}{{\eulercal O}}
\newcommand{\ecP}{{\eulercal P}}
\newcommand{\ecZ}{{\eulercal Z}}
\DeclareMathAlphabet{\beulercal}{U}{eus}{b}{n}
\newcommand{\becD}{{\beulercal D}}
\newcommand{\becE}{{\beulercal E}}
\newcommand{\becF}{{\beulercal F}}
\newcommand{\becG}{{\beulercal G}}
\newcommand{\mfF}{\mathfrak{F}}
\newcommand{\mfI}{\mathfrak{I}}
\newcommand{\mfJ}{\mathfrak{J}}
\newcommand{\mfK}{\mathfrak{K}}
\newcommand{\scC}{\mathscr{C}}
\newcommand{\scD}{\mathscr{D}}
\newcommand{\scE}{\mathscr{E}}
\newcommand{\scF}{\mathscr{F}}
\newcommand{\scG}{\mathscr{G}}

\newcommand{\scK}{\mathscr{K}}
\newcommand{\scO}{\mathscr{O}}
\newcommand{\scP}{\mathscr{P}}
\newcommand{\scS}{\mathscr{S}}

\newcommand{\eq}{\Leftrightarrow}
\newcommand{\imply}{\Rightarrow}
\newcommand{\longto}{\longrightarrow}
\newcommand{\isom}{\overset{\sim}{\to}}
\newcommand{\longisom}{\overset{\sim}{\longto}}

\DeclareMathOperator{\Pic}{Pic}
\DeclareMathOperator{\Div}{Div}
\DeclareMathOperator{\rg}{rg}

\DeclareMathOperator{\Hom}{Hom}
\DeclareMathOperator{\Min}{Min}
\DeclareMathOperator{\Max}{Max}
\DeclareMathOperator{\Inf}{Inf}
\DeclareMathOperator{\Spec}{Spec}
\DeclareMathOperator{\ddiv}{div}
\DeclareMathOperator{\intrel}{intrel}
\let\leq\leqslant
\let\geq\geqslant
\newcommand{\sumu}[1]{\underset{#1}{\sum}}
\newcommand{\produ}[1]{\underset{#1}{\prod}}
\newcommand{\cupu}[1]{\underset{#1}{\cup}}
\newcommand{\capu}[1]{\underset{#1}{\cap}}
\newcommand{\bigcapu}[1]{\underset{#1}{\bigcap}}
\newcommand{\Minu}[1]{\underset{#1}{\Min}}
\newcommand{\Infu}[1]{\underset{#1}{\Inf}}
\newcommand{\degu}[1]{\underset{#1}{\deg}}
\newcommand{\eps}{\varepsilon}
\newcommand{\vide}{\varnothing}
\newcommand{\eqdef}{\overset{\text{{\tiny{d\'ef}}}}{=}}
\newcommand{\wt}{\widetilde}
\newcommand{\acc}[2]{\left\langle #1\, ,\,#2 \right\rangle} 
\newcommand{\norm}[1]{\left|\left| #1 \right|\right|} 
\newcommand{\abs}[1]{\left| #1 \right|} 
\newcommand{\inv}{\times} 
\newcommand{\ceff}{C_{\text{\textnormal{\tiny{eff}}}}}
\newcommand{\eg}{{\it e.g.}\ }
\newcommand{\struct}[1]{\scO_{#1}}
\newcommand{\card}[1]{\left|#1\right|}
\newcommand{\courbe}{\scC}
\newcommand{\gc}{g_{{}_{\courbe}}}
\newcommand{\hc}{h_{{}_{\courbe}}}
\newcommand{\diveffc}{\Div_{\text{\textnormal{eff}}}(\courbe)}
\newcommand{\OdeC}{\struct{\courbe}}
\newcommand{\TNS}{T_{\text{NS}}}
\DeclareMathOperator{\Cox}{Cox}
\DeclareMathOperator{\Cov}{Cov}
\DeclareMathOperator{\Rlv}{Rlv}
\newcommand{\ecOu}[1]{\underset{#1}{\ecO}}
\newcommand{\fonc}{\ecN}
\newcommand{\wfonc}{\wt{\ecN}}
\newcommand{\cone}[1]{\ecC\left(#1\right)}
 
\newcommand{\cprinc}{c_{\text{\textnormal{pr}}}} 
\newcommand{\Dtot}{\scD_{\text{tot}}}
\newcommand{\scan}[1]{s_{#1}} 
\newcommand{\idex}{\mathscr{I}_X} 
\newcommand{\idexh}{\mathscr{I}^\text{\textnormal{hom}}_{X}} 
\newcommand{\Zerr}[1]{Z_{\text{\textnormal{err}},#1}}
\newcommand{\Zpr}[1]{Z_{\text{\textnormal{princ}},#1}}
\newcommand{\Zprgen}{Z_{\text{\textnormal{princ}}}}
\newcommand{\Zprprgen}{Z^{\,0}_{\text{\textnormal{princ}}}}
\newcommand{\can}[1]{\scK_{#1}}
\newcommand{\indsec}{\mfI}
\newcommand{\jndsec}{\mfJ}
\newcommand{\kndsec}{\mfK}
\newcommand{\classadm}{\ecC}
\newcommand{\classadminc}{\classadm_{\text{inc}}}
\newcommand{\indfan}{\cI}
\newcommand{\jndfan}{\cJ}
\newcommand{\kndfan}{\cK}
\DeclareMathOperator{\fact}{fact}
\DeclareMathOperator{\dens}{dens}
\newcommand{\ie}{{\it i.e.}\ }
\newcommand{\cf}{{\it cf.}\  }
\title[La conjecture de Manin]{La conjecture de Manin g\'eom\'etrique pour une famille
de quadriques intrins\`eques}
\author{David Bourqui}
\address{I.R.M.A.R\\Campus de Beaulieu\\
35042 Rennes cedex \\ France}
\email{david.bourqui@univ-rennes1.fr}
\date{}
\begin{document}
\frontmatter
\begin{abstract}
Nous \'etablissons une version de la
  conjecture de Manin pour une famille de quadriques intrins\`eques, le corps de base \'etant un corps global de
  caract\'eristique positive. Nous expliquons \'egalement comment une
tr\`es l\'eg\`ere variante de la m\'ethode employ\'ee permet d'\'etablir cette m\^eme conjecture
pour une certaine surface de del Pezzo g\'en\'eralis\'ee.
\end{abstract}
\begin{altabstract}
We prove a version of Manin's conjecture 
for a certain family of intrinsic quadrics, the base
field being a global field of positive characteristic.
We also explain how a very slight variation of the method we use allows to 
establish the conjecture for a certain generalized del Pezzo surface.
\end{altabstract}
\altkeywords{Manin's conjecture, height zeta function, global field of
positive characteristic, universal torsors, Cox rings}
\maketitle
Dans ce texte, nous \'etablissons une version de la
conjecture de Manin sur le comportement asymptotique du nombre de
courbes de degr\'e 
anticanonique born\'e 
pour une certaine famille de quadriques intrins\`eques, \ie de vari\'et\'es
dont l'anneau total de coordonn\'ees s'identifie \`a l'anneau de coordonn\'ees d'une
quadrique affine. Cette famille est construite \`a l'aide des r\'esultats
de \cite{BerHau:Cox} qui permettent de b\^atir des vari\'et\'es
d'anneaux totaux de coordonn\'ees fix\'es (on renvoie \'egalement \`a cette
r\'ef\'erence pour la justification du terme \og intrins\`eque\fg~;
soulignons qu'une quadrique intrins\`eque n'est pas en g\'en\'eral isomorphe
\`a une quadrique).
Le r\'esultat principal de cet article est le suivant 
(on se reportera au th\'eor\`eme \ref{thm:princ}
pour un \'enonc\'e plus pr\'ecis).
\begin{thm}
Soit $k$ un corps fini, $\courbe$ une $k$-courbe projective et lisse
et  $(X_n)_{n\geq 3}$ la famille de $k$-vari\'et\'es d\'efinie \`a la sous-section 
\ref{subsec:princ}. Alors pour tout $n\geq 3$ la conjecture de Manin sur le comportement
asymptotique du nombre de morphismes de $\courbe$ vers $X_n$ de degr\'e anticanonique born\'e vaut
pour $X_n$.
\end{thm}
Il est \`a noter qu'on a $\dim(X_n)=n-1$ et que $X_3$ est isomorphe au
plan projectif \'eclat\'e en trois points align\'es, trait\'e dans un
pr\'ec\'edent travail (\cf \cite{Bou:compt}). Nous expliquons \'egalement \`a la fin
de l'article comment une l\'eg\`ere variante de la m\'ethode employ\'ee permet
d'obtenir le r\'esultat suivant.
\begin{thm}\label{thm:dP6A2}
On conserve les notations de l'\'enonc\'e du th\'eor\`eme pr\'ec\'edent.
Soit $X$ la d\'esingularisation minimale de la surface de
del Pezzo singuli\`ere de degr\'e $6$  avec une
singularit\'e de type $A_2$. 
Alors la conjecture de Manin sur le comportement
asymptotique du nombre de morphismes de $\courbe$ vers $X$ de degr\'e anticanonique born\'e vaut
pour $X$.
\end{thm}
Le principe g\'en\'eral de la d\'emonstration suit celui de
\cite{Bou:compt}. L'id\'ee de d\'epart, d\^ue \`a Salberger (\cite{Sal:tammes})
et Peyre (\cite{Pey:duke}), consiste \`a exploiter l'existence d'un certain torseur au-dessus de la vari\'et\'e.  
Comme expliqu\'e dans \cite{Bou:compt}, ceci permet de r\'e\'ecrire la fonction z\^eta des hauteurs en termes des relations
d\'efinissant l'anneau de coordonn\'ees total et de s\'eries g\'en\'eratrices index\'ees par les points entiers du dual du c\^one effectif.
Afin de d\'egager terme dominant et termes non significatifs de la
fonction z\^eta des hauteurs, il s'agit alors de d\'ecomposer les s\'eries obtenues suivant les
\og bonnes \fg\ et \og mauvaises\fg\ r\'egions du c\^one
effectif dual.
Pour faciliter une adaptation future de la m\'ethode employ\'ee \`a d'autres vari\'et\'es ou
familles de vari\'et\'es, nous commen\c cons par d\'ecrire  la strat\'egie
suivie pour une famille 
plus g\'en\'erale que celle pour laquelle le r\'esultat sera finalement
d\'emontr\'e, d\'egageant au passage des hypoth\`eses suffisantes pour que la
m\'ethode aboutisse. La classe que nous \'etudions est consitu\'ee des
hypersurfaces intrins\`eques pour lesquelles la d\'ependance en les
param\`etres de l'\'equation de l'anneau de Cox est lin\'eaire
(\cf la remarque \ref{rem:param}, notamment pour le sens de \og
param\`etres\fg).

Une des am\'eliorations par rapport \`a
\cite{Bou:compt} est que les d\'ecompositions li\'ees au c\^one effectif
sont exprim\'ees de mani\`ere intrins\`eque. 
En particulier, bien que le c\^one effectif des vari\'et\'es consid\'er\'es
dans les \'enonc\'es ci-dessus soit 
simplicial, cette particularit\'e ne joue aucun r\^ole dans
notre d\'emonstration. Par contre, on verra clairement que la position dans le c\^one
effectif des diviseurs des sections globales utilis\'ees pour engendrer
l'anneau total de coordonn\'ees a une influence cruciale pour la mise en oeuvre de
la m\'ethode. On renvoie \`a la remarque \ref{rem:hyp} pour quelques commentaires sur la nature des
hypoth\`eses mises en jeu.

Pour conclure cette introduction, soulignons que, bien que notre
m\'ethode n'utilise absolument pas cette structure, les vari\'et\'es 
consid\'er\'ees dans les \'enonc\'es ci-dessus
sont des compactifications \'equivariantes de l'espace affine (\cf
la remarque \ref{rem:comp}). L'analogue de nos r\'esultats pour les corps de nombres
d\'ecoule donc d'un th\'eor\`eme plus g\'en\'eral de Chambert-Loir et
Tschinkel sur la validit\'e des conjectures de Manin pour les
compactifications \'equivariantes d'espaces affines d\'efinies sur un
corps de nombres (\cite{CLT:inv}). L'adaptation de leur m\'ethode
au cas d'un corps global de caract\'eristique positive
est probablement faisable mais reste \`a mettre en \oe uvre.

Nous d\'ecrivons \`a pr\'esent bri\`evement l'organisation de l'article.
La section \ref{sec:fzh} contient les rappels utiles.
La strat\'egie g\'en\'erale de d\'emonstration est d\'ecrite dans la section
\ref{sec:hyp:int}. Elle est synth\'etis\'ee par le th\'eor\`eme \ref{thm:synth}.
La construction de notre famille de quadriques intrins\`eques fait
l'objet du d\'ebut de la  section
\ref{sec:constr} et la v\'erification des hypoth\`eses ad hoc pour cette
famille occupe la fin de cette m\^eme section. 
Enfin la derni\`ere section explique l'adaptation de la m\'ethode
permettant d'obtenir le th\'eor\`eme \ref{thm:dP6A2}.

\section{Fonction z\^eta des hauteurs et rel\`evement
  au torseur universel}\label{sec:fzh}
Dans cette section, tout en fixant quelques notations, nous rappelons
le contexte de notre \'etude et le r\'esultat de rel\`evement de la fonction
z\^eta des hauteurs au torseur universel d\'emontr\'e dans \cite{Bou:compt}.
\subsection{Fonction z\^eta des hauteurs et conjecture de Manin g\'eom\'etrique}\label{subsec:fonctionzeta}
Soit $k$ un corps fini de cardinal $q$. 
Soit $\courbe$ une $k$-courbe projective, lisse et g\'eom\'etriquement
int\`egre, dont on note $\gc$ le genre et $\hc$ le nombre de classes
de diviseurs de degr\'e $0$. 
On note $\courbe^{(0)}$ l'ensemble des
points ferm\'es de $\courbe$. 
Pour $v\in \courbe^{(0)}$, on note
$\kappa_v$ le corps r\'esiduel et $q_v=q^{\,f_v}$ son cardinal.
On note $\diveffc$ le monoïde des diviseurs effectifs de $\courbe$.
Rappelons que si $\ecD$ est un diviseur de $\courbe$ la dimension $\ell(\ecD)\eqdef\dim(H^0(\courbe,\OdeC(\ecD)))$
est toujours major\'ee par $1+\deg(\ecD)$ et vaut $1-\gc+\deg(\ecD)$ si
on a $\deg(\ecD)\geq 2\,\gc-1$.
En particulier, on a pour tout $d\geq 0$ la majoration
\begin{equation}\label{eq:maj:D}
\card{\{\ecD\in \diveffc,\,\deg(\ecD)=d\}}\leq \frac{q^{\,1+d}\,\hc}{q-1}.
\end{equation}

Soit $X$ une $k$-vari\'et\'e projective, lisse 
et g\'eom\'etriquement int\`egre d\'efinie sur $k$. 
On suppose que son groupe de Picard g\'eom\'etrique est libre de rang fini
et d\'eploy\'e, \ie que l'action du groupe de Galois absolu est triviale.
On note $\can{X}$
la classe du faisceau canonique de $X$ dans le groupe de Picard. On
suppose qu'elle est situ\'ee \`a l'int\'erieur du c\^one effectif $\ceff(X)$ de $X$.

Pour $U$ ouvert de Zariski
non vide de $X$ assez petit et $n\geq 0$, on note $N(X,-\can{X},U,n)$
le nombre de $k$-morphismes $f\,:\,\courbe\to X$ dont l'image recontre
$U$ et de degr\'e anticanonique $n$. Si $U$ est assez petit,
$N(X,-\can{X},U,n)$ est fini pour tout $n$ et 
on peut d\'efinir la fonction z\^eta des
hauteurs anticanonique comme l'\'el\'ement de $\bZ[[t]]$ suivant~:
\begin{equation}
Z_{X,-\can{X},U}(t)\eqdef \sum_{n\geq 0}N(X,-\can{X},U,n)\,t^n.
\end{equation}
Voici une version de la conjecture de Manin dans ce cadre.
\begin{question}\label{ques:manin}
Soit $\delta\eqdef \Max \{d\in \bN_{>0},\, -\can{X}\in
d\,\Pic(X)\}$
et $\wt{Z}_{X,-\can{X},U}(t)$ la s\'erie telle que
$\wt{Z}_{X,-\can{X},U}(t^{\,\delta})=Z_{X,-\can{X},U}(t)$.
Est-il vrai que si $U$ est assez petit
la s\'erie $Z_{X,-\can{X},U}(t)$
a pour rayon de convergence  $q^{-\delta}$ et que 
pour un certain $\eps>0$ 
sa somme 
se prolonge en une
fonction m\'eromorphe sur le disque $\abs{t}<q^{-\delta+\eps}$ ayant un 
p\^ole d'ordre $\rg(\Pic(X))$ en $t=q^{-\delta}$, 
et des p\^oles d'ordre
au plus $\rg(\Pic(X))-1$ en tout autre point du cercle de
rayon $q^{-\delta}$, et v\'erifiant
\begin{equation}
\lim_{t\to q^{-\delta}}\left(t-q^{\,-\delta}\right)^{\rg(\Pic(X))}\,
\wt{Z}_{X,-\can{X},U}(t)
=\alpha(X)\,\gamma(X)
\end{equation}
o\`u
\begin{equation}
\alpha(X)
\eqdef
\lim_{t\to 1} \,\,\,\,(1-t)^{\rg(\Pic(X))}\!\!\!\!\!\!
\sum_{y\in \ceff(X)^{\vee}\cap \Pic(X)^{\vee}}\,t^{\,\acc{y}{-\can{X}}}
\end{equation}
et
\begin{equation}
\gamma(X)
\eqdef
\left(\frac{\hc\,q^{(1-\gc)}}{q-1}\right)^{\rg(\Pic(X))}\:
q^{(1-\gc)\,\dim(X)}\,
\prod_{v\in\courbe^{(0)}}
(1-q_v^{-1})^{\rg(\Pic(X))}\,\frac{\card{X(\kappa_{v})}}{q_v^{\,\dim(X)}}.
\end{equation}
\end{question}

\begin{defi}
\begin{enumerate}
\item
Soit $(a_n)\in \bC^{\bN}$ et $(b_n)\in \bR^{\bN}$. On
dit que la s\'erie $\sum a_nt^n$ est major\'ee par la s\'erie $\sum
b_n\,t^n$ si on a $\abs{a_n}\leq b_n$ pour tout $n$.
\item
Soit $(a_n)\in \bC^{\bN}$, $k\geq 1$ un entier et $\rho>0$ un r\'eel.
On dit que la s\'erie $\sum a_nt^n$ est $\rho$-contr\^ol\'ee \`a l'ordre $k$
si elle v\'erifie l'une des deux  conditions 
 suivantes (\'equivalentes d'apr\`es les estim\'ees de Cauchy)~:
\begin{enumerate}
\item
on a 
$
a_n=\ecOu{n\to +\infty}\left(n^{k-1}\,\rho^{-n}\right)
$~;
\item
la s\'erie $\sum a_n\,t^n$ est major\'ee par une s\'erie 
dont le rayon de convergence est sup\'erieur \`a $\rho$
et dont la somme se prolonge en une fonction m\'eromorphe
sur un disque de rayon strictement sup\'erieur  \`a $\rho$,
ayant des p\^oles d'ordre au plus $k$ sur le cercle de rayon $\rho$.
\end{enumerate}
\end{enumerate}
\end{defi}
\begin{lemme}\label{lm:control}
On conserve les hypoth\`eses et notations pr\'ec\'edentes.
Si $U$ est un ouvert de $X$ tel que la s\'erie
\begin{equation}
Z_{X,-\can{X},U}(t)-
\gamma(X)\!\!\!\!\!
\sum_{y\in \ceff(X)^{\vee}\cap \Pic(X)^{\vee}}\,(q\,t)^{\,\acc{y}{-\can{X}}}
\end{equation}
est $q^{-1}$-contr\^ol\'ee \`a l'ordre $\rg(\Pic(X))-1$,
alors la r\'eponse \`a la question \ref{ques:manin} est positive pour $U$.
\end{lemme}

\subsection{Anneau de Cox, inversion de M\"obius, et rel\`evement au  torseur universel}\label{subsec:Cox}
On conserve les notations et hypoth\`eses de la section pr\'ec\'edente.
On suppose en outre que l'anneau total de coordonn\'ees  (ou anneau de
Cox) de $X$ (\cf \eg \cite{Has:eq:ut:cox:rings}),
not\'e $\Cox(X)$, est de type fini.
Soit $\{u_i\}_{i\in \indsec }$ 
une famille finie de sections globales (non constantes) 
qui engendre $\Cox(X)$. 
Soit $\idex$ l'id\'eal $\Pic(X)$-homog\`ene noyau du morphisme naturel
$k[u_i]_{i\in \indsec }\to \Cox(X)$ et $\idexh$ l'ensemble de ses \'el\'ements homog\`enes.
Pour $i\in \indsec $, on note
$\scE_i$ le diviseur des z\'eros de $u_i$.
Soit $X_0$ l'ouvert de $X$ \'egal au compl\'ementaire de la r\'eunion des
diviseurs $\{\scE_i\}_{i\in \indsec }$.

Soit $\TNS(X)\eqdef \Hom(\Pic(X),\bG_m)$ le tore de N\'eron-Severi.
Soit $\hat{X}$ l'ouvert de $\Spec(\Cox(X))$ form\'e des points
semi-stables vis-\`a-vis de la $\TNS(X)$-lin\'earisation sur le fibr\'e trivial de
$\Spec(\Cox(X))$ induite par le choix d'une classe ample.
Le quotient g\'eom\'etrique de $\hat{X}$ par $\TNS(X)$ existe et s'identifie
naturellement \`a $X$. On montre en outre que $\hat{X}\to X$ repr\'esente
l'unique classe de torseurs universels au-dessus de $X$ (\cf \cite{Has:eq:ut:cox:rings,hukeel:mori}).

Une classe $\classadm$ de parties de $\indsec $ 
sera dite \termin{admissible} si on a
\begin{equation}
\hat{X}=\Spec(\Cox(X))\cap \left(\cupu{\jndsec \in \classadm } \prod_{i\in \jndsec }u_i\neq 0 \right).
\end{equation}
Par exemple, soit $\scD\in \Pic(X)$ une classe ample et 
$\classadm $ l'ensemble des parties $\jndsec $ de $\indsec$ telles qu'il existe une famille
$(\lambda_i)_{i\in \jndsec}$ de rationnels strictement positifs v\'erifiant
$
\scD=\sumu{i\in \jndsec }\lambda_i\,\scE_i.
$
Alors $\classadm $ est admissible.
\begin{rem}\label{rem:inter:0}
Pour tout $i\in \indsec $, l'image r\'eciproque du diviseur $\scE_i$ par
le morphisme quotient $\hat{X}\to X$ est $\hat{X}\cap \{u_i=0\}$.
On en d\'eduit que si une partie $\jndsec$ de $\indsec $ v\'erifie 
$\cap_{i\in  \indsec }\,\,\scE_i\neq \vide$, toute classe admissible contient une
partie $\kndsec$ telle que $\kndsec\cap \jndsec =\vide$. On en d\'eduit
m\^eme que la classe
\beq
\classadminc\eqdef\{\jndsec \subset \indsec,\quad \capu{i\notin \jndsec}\ecE_i\neq \vide\}
\eeq
est admissible.
\end{rem}
On a la g\'en\'eralisation classique suivante de la formule d'inversion de M\"obius
\begin{prop}\label{prop:mu}
Il existe une unique fonction $\mu_{X}\,:\,\diveffc^{\indsec }\longto \bC$
v\'erifiant
\begin{equation}
\forall \,\becD\in \diveffc^{\indsec },\quad
\sum_{0\leq \becE\leq \becD}
\mu_{X}(\becE)
=
\left\{
\begin{array}{cl}
1 &\text{si }\underset{\jndsec \in \classadminc }{\Inf}\left(\sum_{i\in \jndsec }\ecD_i\right)=0\\
0 &\text{sinon}
\end{array}
\right.
\end{equation}
Cette fonction v\'erifie en outre les propri\'et\'es suivantes~:
\begin{enumerate}
\item
 elle est multiplicative, c'est-\`a-dire que si $\becE$ et $\becE'$ v\'erifient 
\begin{equation}
\forall i\in \indsec,\quad \Inf(\ecE_i,\ecE'_i)=0
\end{equation}
alors on a
\begin{equation}
\mu_{X}(\becE+\becE')=\mu_{X}(\becE)\,\mu_{X}(\becE')\quad ;
\end{equation}
\item
pour tout $v\in\courbe^{(0)}$ et tout $\balpha\in  \bN^{\,\indsec}$,  
$\mu_{X}(\,(\alpha_{i}\,v)\,)$ ne d\'epend que de $\balpha$~; 
on note $\mu_{X}^0(\balpha)$ cette valeur. 
Pour tout $\balpha\in \bN^{\indsec}$, on a
\begin{equation}
\sum_{0\leq \bbeta\leq \balpha}\mu_{X}^0(\bbeta)
=
\left\{
\begin{array}{cl}
1 &\text{si }\capu{i,\,\,\alpha_i\neq 0}\scE_i\neq \vide\\
0 &\text{sinon}.
\end{array}
\right.
\end{equation}
En particulier, on a $\mu_{X}^0(\balpha)=0$ dans les cas suivants~:
\begin{enumerate}
\item
il existe $i\in \indsec $ tel que $\alpha_{i}\geq 2$~;
\item
$\balpha$ est non nul et l'intersection 
$\cap_{i,\,\,\alpha_i\neq 0}\,\scE_i$ est non vide~; 
ceci vaut en particulier si
on a 
$\sum_{i\in \indsec }\alpha_i=1$.
\end{enumerate}
\end{enumerate}
\end{prop}
On \'ecrit \`a pr\'esent $\indsec =I\sqcup J$, o\`u $I$ est tel que les classes 
$\{\ecE_i\}_{i\in I}$ forment une base de $\Pic(X)$.
Pour $i\in I$ (respectivement $j\in J$), on notera $u_i=s_i$
(respectivement $u_j=t_j$)
et $\scE_i=\scF_i$ 
(respectivement $\scE_j=\scG_j$).
La notation $\becE\in \diveffc^{I\cup J}$ d\'esignera toujours un couple $(\becF,\becG)$
o\`u $\becF\in \diveffc^I$ et $\becG\in \diveffc^J$.
\'Ecrivons, pour $j\in J$,
\begin{equation}
\scG_j=\sum_{i\in I} a_{i,j}\,\scF_i,\quad
a_{i,j}\in \bZ.
\end{equation}
Pour $\ecD\in \diveffc$, on note $\scan{\ecD}$ la section canonique de $\OdeC(\ecD)$.
Pour $\becD\in \diveffc^I$ et $\becE\in \diveffc^{I\cup J}$,
on d\'esigne par $\fonc(\becD,\becE)$ le cardinal de l'ensemble
des \'el\'ements $(t_j)_{j\in J}$ {\em tous non nuls} du produit
\begin{equation}
\prod_{j\in J} 
H^0(\courbe,\struct{\courbe}
(-\ecG_j+\sum_{j}a_{i,j}\,(\ecD_j+\ecF_i)))
\end{equation}
v\'erifiant les relations
\begin{equation}
\forall F\in \idexh,\quad
F(\scan{\ecD_i}\,\scan{\ecF_i},\,t_j\,\scan{\ecG_j})=0.
\end{equation}
Des r\'esultats des sections  1.3, 1.4 et 1.6 de \cite{Bou:compt} d\'ecoule alors
la formule suivante, que l'on peut voir comme une formule de
rel\`evement de la fonction z\^eta des hauteurs au torseur universel~:
\begin{equation}\label{eq:form:relevement}
Z_{X,-\can{X},X_0}(t)
=
\sum_{\becE\in \diveffc^{I\cup J}}
\mu_X(\becE)
\sum_{
\substack{
y\in\Pic(X)^{\vee}\cap \ceff(X)^{\vee}
\\
\acc{y}{\scG_j}\geq \deg(\ecG_j),\quad j\in J
\\
\acc{y}{\scF_i}\geq \deg(\ecF_i),\quad i\in I
\\
\becD\in \diveffc^I\\
\deg(\ecD_i)=\acc{y}{\scF_i}-\deg(\ecF_i),\quad i\in I 
}
}
\fonc(\becD,\becE)
\,\,
t^{\acc{y}{-\can{X}}}.
\end{equation}
\begin{rem}\label{rem:param}
Du point de vue du comptage, le choix de la d\'ecomposition
$\indsec =I\sqcup J$ revient grosso modo \`a fixer les variables
$\{s_j\}_{i\in I}$, les variables $\{t_j\}_{j\in J}$ devenant des
param\`etres.
La situation que nous allons consid\'erer ci-dessous est celle o\`u
l'anneau de Cox n'a qu'une relation qui d\'epend en outre lin\'eairement
de ces param\`etres~; en un sens, il s'agit donc de la situation la plus
simple apr\`es celle des vari\'et\'es toriques (o\`u il n'y a pas de
relations, l'anneau de Cox \'etant polynomial).
\end{rem}

\section{Le cas de certaines hypersurfaces intrins\`eques}\label{sec:hyp:int}
On reprend les hypoth\`eses et notations 
de la section pr\'ec\'edente.
On suppose en outre qu'on a un isomorphisme
\begin{equation}
\Cox(X)\longisom k[(s_i)_{i\in I},(t_j)_{j\in J}]/F(s_i,t_j).
\end{equation}
o\`u $F$ est $\Pic(X)$-homog\`ene de degr\'e $\Dtot$ et est {\em lin\'eaire en les
  $t_j$},
i.e. s'\'ecrit
\begin{equation}
F=\sum_{j\in J} t_j\,\prod_{i\in I}s_i^{b_{i,j}},\quad b_{i,j}\in \bN.
\end{equation}
D'apr\`es \cite[proposition 8.5]{BerHau:Cox}, on a  la formule d'adjonction suivante.
\begin{lemme}
La classe du fibr\'e anticanonique est
\begin{equation}\label{eq:adj}
-\can{X}=\sum_{i\in I}\scF_i+\sum_{j\in  J}\scG_j-\Dtot.
\end{equation}
\end{lemme}

Dans cette section nous d\'ecrivons le sch\'ema d'une strat\'egie pour \'etablir que dans
ce cas la question \ref{ques:manin} a une r\'eponse positive pour $U=X_0$.
Plus pr\'ecisons nous montrons le r\'esultat suivant~:
\begin{thm}\label{thm:synth}
On conserve les notations et hypoth\`eses pr\'ec\'edentes.

On suppose que 
les hypoth\`eses 
\ref{hyp:conv:cEeps},
\ref{hyp:rel:term:princ},
\ref{hyp:zerr2:bis}
et \ref{hyp:pos}
d\'ecrites ci-dessous
sont satisfaites.

Alors la s\'erie
\begin{equation}
Z_{X,-\can{X},X_0}(t)-
\gamma(X)\!\!
\sum_{y\in \ceff(X)^{\vee}\cap \Pic(X)^{\vee}}\,(q\,t)^{\,\acc{y}{-\can{X}}}
\end{equation}
est $q^{-1}$-contr\^ol\'ee \`a l'ordre $\rg(\Pic(X))-1$,
\end{thm}
La d\'emonstration de ce th\'eor\`eme
occupe le reste de cette section.
\begin{rem}\label{rem:hyp}
Les hypoth\`eses que l'on va d\'egager sont de deux
natures diff\'erentes~: d'une part, on demande que les diviseurs des
z\'eros des sections $t_j$ soient \og suffisamment positifs\fg~; d'autres part,
qu'une certaine s\'erie g\'en\'eratrice de nature combinatoire naturellement 
associ\'ee \`a la relation d\'efinissant l'anneau de Cox ait de bonnes
propri\'et\'es analytiques, et que son \og terme dominant\fg\  soit reli\'e au
nombre de points de $X$. 

Les surfaces de del Pezzo g\'en\'eralis\'ees qui sont des hypersurfaces
intrins\`eques ont \'et\'e classifi\'ees par Derenthal dans
\cite{Der:sdp:ut:hyp}. 
Un certain nombre d'entre elles sont de la forme consid\'er\'ee ici. 
Il appara\^\i t malheureusement que parmi ces derni\`eres la seule qui satisfasse nos
hypoth\`eses de positivit\'e est le plan projectif \'eclat\'e en trois
points align\'es. La fin de cette article explique cependant comment une l\'eg\`ere variante de la
m\'ethode d\'ecrite dans cette section s'applique \`a un autre membre
de la liste de Derenthal, pour lequel la d\'ependance en les
param\`etres $\{t_j\}_{j\in J}$ est \og presque\fg lin\'eaire. Au prix d'un
certain nombre de  complications techniques, la d\'emarche devrait s'av\'erer fructueuse
pour deux autres membres de la liste. C'est l'objet d'un travail en cours.

Vis-\`a-vis de la m\'ethode d\'ecrite ci-dessous, le fait que les hypoth\`eses
de positivit\'e ne sont pas satisfaites signifie que l'on englobe dans
le terme d'erreur 
des termes correspondant \`a des r\'egions trop importantes du c\^one
effectif dual, 
dont on ne peut en fait plus garantir qu'ils ne contribuent pas au terme principal de la
fonction z\^eta des hauteurs. Un rel\^achement des hypoth\`eses de 
positivit\'e devrait passer par une am\'elioration significative du lemme
de comptage \ref{prop:compt}. 
Il serait sans doute   int\'eressant \`a cet \'egard de pouvoir \'etendre la validit\'e
de la m\'ethode au cas de la 
d\'esingularisation minimale de la surface de del
Pezzo singuli\`ere de degr\'e 5 avec singularit\'e de type $A_1$. Pour cette
quadrique intrins\`eque, l'\'equation de l'anneau de Cox est tr\`es
similaire au cas du plan projectif \'eclat\'e en trois points align\'es,
mais la configuration du c\^one effectif est radicalement diff\'erente.
Il est \`a noter que pour cet exemple (comme pour d'autres) nous avons pu v\'erifier
que les hypoth\`eses du second type \'etaient satisfaites (\cf la remarque \ref{rem:eff}).
Il serait int\'eressant de d\'egager une d\'emonstration conceptuelle de ce genre de
r\'esultat, notamment en ce qui concerne la v\'erification de l'hypoth\`ese \ref{hyp:rel:term:princ}. 
\end{rem}

\subsection{Quelques lemmes pr\'eliminaires}
\begin{lemme}\label{lm:estim}
Soit $n\geq 1$ un entier, et $\rho>0$ un r\'eel.
Soit $(a_{\bd})\in \bC^{\bN^n}$ et 
\begin{equation}
F(\bt)\eqdef\sum_{\bd\in \bN^n}a_{\bd}\,\prod_{1\leq i\leq n}t_i^{d_i}.
\end{equation}
On suppose que $F(\bt)$ converge absolument sur un polydisque de rayon
$(r,\dots,r)$ avec $r>\rho^{-1}$.
Soit $(b_{\bd})\in \bC^{\bN^n}$ d\'efinie par
\begin{equation}
\frac{F(\bt)}
{(1-\rho\,t_1)\dots(1-\rho\,t_n)}
=\sum_{\bd\in \bN^n}b_{\bd}\,t_1^{d_1}\dots t_n^{d_n}.
\end{equation}
Alors on a pour tout $\eps>0$ assez petit 
\begin{equation}
\forall \bd\in \bN^n,\quad
\abs{b_{\bd}-F(\rho^{-1},\dots,\rho^{-1})\,\rho^{\abs{\bd}}}
\leq 
\frac{\rho^{-\eps}\,\norm{F}_{\rho^{-1+\eps}}}{(1-\rho^{-\eps})^{n}}\sum_{1\leq i\leq n}
\rho^{(1-\eps)d_i+\sum_{j\neq i} d_j}
\end{equation}
o\`u $\norm{F}_{\eta}\eqdef \Max_{\abs{t_i}=\eta}\abs{F(\bt)}$.
\end{lemme}
\begin{proof}
D'apr\`es les estimations de Cauchy, on a pour $\eps>0$ assez petit
\begin{equation}\label{eq:in:cauchy}
\forall \bd\in \bN^n,\quad \abs{a_{\bd}}\leq \norm{F}_{\rho^{-1+\eps}}\,\rho^{(1-\eps)\abs{\bd}}.
\end{equation}
Par ailleurs un calcul imm\'ediat montre qu'on a 
\begin{multline}
F(t_1,\dots,t_k,\rho^{-1},\dots,\rho^{-1})
-F(t_1,\dots,t_{k-1},\rho^{-1},\dots,\rho^{-1})\\
=(\rho\,t_k-1)
\sum_{
\substack{
\bd\in \bN^k,\,\delta\in \bN,\\
(\delta_j)_{k+1\leq j\leq n}\in \bN^{n-k}
}}
a_{d_1,\dots,d_k+1+\delta,\delta_{k+1},\dots \delta_n}
\rho^{-(\delta+1)-\sumu{k+1\leq j\leq n}\delta_{j}}
\prod_{1\leq i\leq
  k}t_i^{d_i}.
\end{multline}
Si on pose 
\begin{equation}
\frac{F(t_1,\dots,t_k,\rho^{-1},\dots,\rho^{-1})
-F(t_1,\dots,t_{k-1},\rho^{-1},\dots,\rho^{-1})}
{\prod_{1\leq i\leq n}(1-\rho\,t_i)}
\eqdef G_k(\bt)\eqdef \sum_{\bd\in \bN^n}b_{k,\bd}\prod_{1\leq i\leq n} t_i^{d_i}
\end{equation}
on a donc pour tout $\bd\in \bN^n$ la relation
\begin{equation}
b_{k,\bd}=-\sum_{\substack{0\leq \delta_i\leq d_i,\,1\leq i\leq k-1\\
\delta\in \bN,\\
\delta_j\in \bN,\,
k+1\leq j\leq n
}}
a_{\delta_1,\dots,\delta_{k-1},d_k+1+\delta,\delta_{k+1},\dots \delta_n}
\rho^{-(\delta+1)-\sumu{k+1\leq j\leq n}\delta_{j}+\sumu{1\leq i\leq
    k-1}(d_i-\delta_i)+\sumu{k+1\leq j\leq n}d_j}.
\end{equation}
D'apr\`es \eqref{eq:in:cauchy}, on a pour $\eps>0$ assez petit la majoration
\begin{align}
\abs{b_{k,\bd}}&\leq \norm{F}_{\rho^{-1+\eps}}\,\rho^{-\eps+(1-\eps)\,d_k+\sumu{1\leq i\leq n,\,i\neq k}d_i}
\sum_{\substack{
0\leq \delta_i\leq d_i,\,1\leq i\leq k-1\\
\delta\in \bN,\\
\delta_j\in \bN,\,
k+1\leq j\leq n
}}
\rho^{-\eps(\delta_1+\dots+\delta_{k-1}+\delta+\delta_{k+1}+\dots+\delta_n)}
\\
&\leq 
\frac{\norm{F}_{\rho^{-1+\eps}}\,\rho^{-\eps+(1-\eps)\,d_k+\sum_{1\leq
      i\leq n,\,i\neq k}d_k}}
{(1-\rho^{-\eps})^{n}}.
\end{align}
Compte tenu de la relation
\begin{equation}
\frac{F(\bt)
-F(\rho^{-1},\dots,\rho^{-1})}
{\prod_{1\leq i\leq n}(1-\rho\,t_i)}
=\sum_{1\leq k\leq n} G_k(\bt)
\end{equation}
on obtient bien la majoration annonc\'ee.
\end{proof}
\begin{lemme}\label{lm:prelim}
Soit $N$ un $\bZ$-module libre de rang fini, $\ecC$ un c\^one polyedral
rationnel de dimension maximale de $N\otimes \bR$, $x_0$ un \'el\'ement de l'int\'erieur
de $\cC$ et $x_1$ un \'el\'ement \emph{non nul} de $\ecC$. 
\begin{enumerate}
\item
Pour tout r\'eel $\rho>1$ 
la s\'erie
\begin{equation}
\sum_{
y\in \ecC^{\vee}\cap N^{\vee}
}
\rho^{-\acc{y}{x_1}}\,t^{\acc{y}{x_0}}
\end{equation}
est $1$-contr\^ol\'ee \`a l'ordre $\dim(\ecC)-1$.
\item
Soit $(a_n)_{n\geq 0}$ et $(M_n)_{n\geq 0}$ deux suites de r\'eels positifs.
On suppose qu'il existe $\rho>1$ tel que la s\'erie
$
\sum a_n\,\rho^{M_n}
$
soit convergente.

Alors la s\'erie
\begin{equation}\label{eq:sum:x1}
\sum  
a_n 
\sum_{
\substack{
y\in \ecC^{\vee}\cap N^{\vee}
\\
\acc{y}{x_1}\leq M_n
}}
t^{\acc{y}{x_0}}
\end{equation}
est $1$-contr\^ol\'ee \`a l'ordre $\dim(\ecC)-1$.
\end{enumerate}
\end{lemme}
\begin{proof} Soit $\Delta$ un \'eventail r\'egulier de support
  $\ecC$. Pour tout c\^one $\delta$ de $\Delta$, soit $\{y_{\ell}\}_{\ell\in
    \delta(1)}$ les g\'en\'erateurs des rayons de $\delta$. 
Soit 
\begin{equation}
\delta(1)_{x_1}\eqdef\{\ell\in \delta(1),\quad \acc{y_{\ell}}{x_1}=0\}.
\end{equation}
Comme $x_1$ est non nul, le cardinal de ce dernier ensemble est major\'e
par $\dim(\ecC)-1$. 
Mais la s\'erie \eqref{eq:sum:x1} se r\'e\'ecrit
\begin{equation}
\sum_{\delta\in \Delta}
\left(
\prod_{
\substack{
\ell\in \delta(1)_{x_1} 
}}
\frac{t^{\acc{y_{\ell}}{x_0}}}{1-t^{\acc{y_{\ell}}{x_0}}}
\right)
\sum_{n} a_n
P_{\delta,n}(t)
\end{equation}
o\`u $P_{\delta,n}(t)$ vaut $1$ si $M_n=0$ et d\'esigne sinon le polyn\^ome
\begin{equation}
\sum_{
\substack{
(n_{\ell})\in \left(\bN_{>0}\right)^{\delta(1)\setminus \delta(1)_{x_1}}
\\
\sum n_{\ell}\,\acc{y_{\ell}}{x_1}\leq M_n
}}
t^{\,\,\sum n_{\ell} \acc{y_{\ell}}{x_0}}.
\end{equation}
Dans ce dernier cas, on a a alors pour tout $\rho\geq 0$
\begin{equation}
P_{\delta,n}(\rho)\leq M_n^{\dim(\ecC)}\,\rho^{\dim(\ecC)\,\Max(\acc{y_{\ell}}{x_0})\,M_n}.
\end{equation}
Ceci montre le deuxi\`eme point.
Une d\'ecomposition analogue permet de montrer le premier point.
\end{proof}
La proposition suivante \'etend les r\'esultats de la section 3.3 de \cite{Bou:compt}.
\begin{prop}\label{prop:compt}
Soit $n\geq 1$ un entier,
soient $\{\ecH_j,\ecH'_j\}_{1\leq j\leq n}$ 
et $\ecH$ des diviseurs de $\courbe$
tels que $\ecH$ et les $\ecH_j+\ecH'_j$ sont deux \`a deux lin\'eairement \'equivalents.
Pour $1\leq j \leq n$, soit $s_j$  une section globale non nulle de $\OdeC(\ecH_j)$.
On fixe des isomorphismes 
\begin{equation}
\OdeC(\ecH_j+\ecH'_j)\isom \OdeC(\ecH) \quad 1\leq j\leq n
\end{equation}
ce qui permet de d\'efinir l'application lin\'eaire
\begin{equation}
\varphi_{\bs}\,:\,\prod_{1\leq j \leq
  n}H^0(\courbe,\OdeC(\ecH'_j))\to H^0(\courbe,\OdeC(\ecH)) 
\end{equation}
qui \`a $(t_j)$ associe $\sum t_j\,s_j$.
On note $\Delta_{\bs}$ la dimension du noyau de $\varphi_{\bs}$.
\begin{enumerate}
\item
On a la majoration
\begin{equation}
\Delta_{\bs}\leq n-1+\left(1-\frac{1}{n}\right)\,\sum_{1\leq j\leq n}\deg(\ecH'_j).
\end{equation}
\item
On a l'une des deux majorations suivantes~:
\begin{equation}\label{eq:maj:bis:ker:3}
\Delta_{\bs}\leq 
n-1
+\degu{1\leq j \leq n}(\Inf[\ddiv(s_j)])
-\deg(\ecH)
+\sumu{1\leq j\leq n} \deg(\ecH'_j)
\end{equation}
ou 
\begin{equation}\label{eq:maj:ter:ker:4}
\Delta_{\bs}\leq n-2+\left(1-\frac{1}{n-1}\right)\sumu{1\leq j\leq n} \deg(\ecH'_j).
\end{equation}
\item
On suppose qu'on a
\begin{equation}\label{hyp:cor:clef:part:gen}
\forall 1\leq j\leq n-1,\quad \deg(\ecH'_j)+\deg(\ecH'_{j+1})\geq
\deg(\ecH)-
\degu{1\leq j \leq n}(\Inf[\ddiv(s_j)])
+2\,\gc-1.
\end{equation}
Alors on a
\begin{equation}
\Delta_{\bs}=(n-1)\,(1-\gc)
+\degu{1\leq j \leq n}(\Inf[\ddiv(s_j)])
-\deg(\ecH)
+\sumu{1\leq j\leq n} \deg(\ecH'_j).
\end{equation}
et l'image de $\varphi_{\bs}$ est constitu\'ee de l'ensemble des multiples de $\Inf(\ddiv(s_i))$.
\end{enumerate}
\end{prop}
\begin{proof}
Commen\c cons par la remarque \'el\'ementaire
suivante~: pour toute partie $K$ de $\{1,\dots,n\}$, on a 
\begin{equation}
\Delta_{\bs}
\leq \Delta_{(s_j)_{j\in K}}
+\sum_{j\notin K} \ell(\ecH'_j)
\leq \Delta_{(s_j)_{j\in K}}
+n-\card{K}+\sum_{j\notin K} \deg(\ecH'_j)
\end{equation}
Ceci entra\^\i ne d\'ej\`a pour tout  $j_0\in \{1,\dots,n\}$
la majoration
\begin{equation}
\Delta_{\bs}\leq n-1+\sum_{j\neq j_0}\deg(\ecH'_j).
\end{equation}
En moyennant sur tous les $j_0$, on obtient le premier point.

Montrons le deuxi\`eme point.
Supposons que pour une certaine permutation des indices on ait
\begin{equation}
\deg(\ecH'_{1}+\ecH'_{2}-\ecH+\Inf[\ddiv(s_{1}),\ddiv(s_{2})])\geq 0
\end{equation}
et
\begin{equation}
\forall k\geq 3,\quad
\deg\left(\ecH'_{k}-\Inf[\ddiv(s_j)]_{1\leq j\leq k-1}
+\Inf[\ddiv(s_j)]_{1\leq j\leq k}]
\right)\geq 0.
\end{equation}
D'apr\`es le point 2. du lemme 3.5 de \cite{Bou:compt}, 
on a la majoration 
\begin{equation}
\Delta_{(s_1,s_2)}
\leq 1+\deg(\ecH'_{1})+\deg(\ecH'_{2})-\deg(\ecH)+\deg(\Inf[\ddiv(s_{1}),\ddiv(s_{2})]).
\end{equation}
Par ailleurs pour tout $k\geq 3$, on a
\begin{align}
\Delta_{(s_1,s_2,\dots,s_k)}\hskip-0.2\textwidth&
\\
&\leq
\Delta_{(s_1,\dots,s_{k-1})}+\ell\left(\ecH'_{k}-\Inf[\ddiv(s_j)]_{1\leq
  j\leq k-1}+\Inf[\ddiv(s_j)]_{1\leq j\leq k}\right)
\\
&\leq \Delta_{(s_1,\dots,s_{k-1})}+1+\deg(\ecH'_k)-
\deg(\Inf[\ddiv(s_j)]_{1\leq  j\leq k-1})+
\deg(\Inf[\ddiv(s_j)]_{1\leq  j\leq k})
\end{align}
De proche en proche, on obtient la majoration
\eqref{eq:maj:bis:ker:3}.

Supposons \`a pr\'esent que tous les couples
$(j_1,j_2)\in \{1,\dots,n\}^2$ 
v\'erifient
\begin{equation}
\deg(\ecH'_{j_1}+\ecH'_{j_2}-\ecH+\Inf[\ddiv(s_{j_1}),\ddiv(s_{j_2})])<0
\end{equation}
D'apr\`es le point 1. du lemme 3.5 de \cite{Bou:compt},
$\varphi_{s_{j_1},s_{j_2}}$ est injective. 
On a donc 
\begin{equation}
\Delta_{\bs}
\leq
n-2+\sum_{j\notin \{j_1,j_2\}}\deg(\ecH'_{j}).
\end{equation}
En moyennant sur tous les couples $(j_1,j_2)$, on obtient la
majoration 
\begin{equation}\label{eq:maj:ker:j1j2}
\Delta_{\bs}\leq n-2+\left(1-\frac{2}{n}\right)\sumu{1\leq j\leq n} \deg(\ecH'_j),
\end{equation}
en particulier \eqref{eq:maj:ter:ker:4} est v\'erifi\'ee.

Supposons enfin qu'il existe $n-1\geq k\geq 2$ et une
permutation des indices telle qu'on ait
\begin{equation}
\deg(\ecH'_{1}+\ecH'_{2}-\ecH+\Inf[\ddiv(s_{1}),\ddiv(s_{2})])\geq 0
\end{equation}
\begin{equation}
\forall 3\leq j\leq k,\quad
\deg\left(\ecH'_{k}-\Inf[\ddiv(s_{j})_{1\leq j\leq k-1}
+\Inf[\ddiv(s_j)]_{1\leq j\leq k}]
\right)\geq 0
\end{equation}
\begin{equation}
\forall j_0\geq k+1,\quad
\deg\left(\ecH'_{j_0}-\Inf[\ddiv(s_1,\dots,s_k)]
+\Inf[\ddiv(s_1,\dots,s_k,s_{j_0})]
\right)<0
\end{equation}

Alors, pour tout $j_0\geq k+1$, on a
$\Delta_{(s_1,s_2,\dots,s_k,s_{j_0})}=\Delta_{(s_1,\dots,s_k)}$
et $\Delta_{\bs}$ est major\'e par 
\begin{equation}
\Delta_{(s_1,\dots,s_k)}+(n-k-1)+\sum_{j\geq k+1,\,\, j\neq j_0}\deg(\ecH'_j)
\end{equation}
soit d'apr\`es ce qui pr\'ec\`ede par
\begin{equation}
k-1+\sum_{1\leq j\leq
  k}\deg(\ecH'_{j})-\deg(\ecH)+\deg(\Infu{1\leq j\leq
  k}[\ddiv(s_{j})]
)+(n-k-1)+\sum_{j\geq k+1,\,\,j\neq j_0}\deg(\ecH'_j).
\end{equation}
En utilisant la majoration
\begin{equation}
\deg(\Inf[\ddiv(s_{1}),\ddiv(s_{2}),\dots,\ddiv(s_{k})])\leq
\frac{1}{k}\sum_{1\leq j\leq k}\deg(\ecH_j)
\end{equation}
on aboutit finalement \`a la majoration
\begin{equation}
\Delta_{\bs}\leq n-2+\left(1-\frac{1}{k}\right)\,\sum_{1\leq j\leq
  k}\deg(\ecH'_{j})+\sum_{j\geq k+1,\,\,j\neq j_0}\deg(\ecH'_j)
\end{equation}
En moyennant sur tous les $j_0\geq k+1$, on obtient la majoration
\begin{equation}
\Delta_{\bs}\leq 
n-2+\left(1-\frac{1}{k}\right)\,\sumu{1\leq j\leq k}\deg(\ecH'_j)+\left(1-\frac{1}{n-k}\right)
\sumu{k+1\leq
    j\leq n} \deg(\ecH'_j).
\end{equation}
Ainsi \eqref{eq:maj:ter:ker:4} est encore v\'erifi\'ee dans ce cas.
Ceci ach\`eve la d\'emonstration du deuxi\`eme point.

Le dernier point est une g\'en\'eralisation imm\'ediate (par r\'ecurrence) du point 3 de
\cite[corollaire 3.6]{Bou:compt}.
\end{proof}

\subsection{D\'ecomposition de la fonction z\^eta des hauteurs}
Pour $K\subset J$, $\becD\in \diveffc^I$ et $\becE\in \diveffc^{I\cup J}$,
on d\'esigne par $\fonc_K(\becD,\becE)$ le cardinal de l'ensemble
\begin{equation}
\{(t_j)_{j\in J}\in\prod_{j\in J} 
H^0(\courbe,\OdeC(-\ecG_j+\sum_{i} a_{i,j}\,(\ecD_i+\ecF_i)))
\end{equation}
v\'erifiant
\begin{equation}
\forall j\notin K,\quad t_j=0
\end{equation}
et
\begin{equation}
\sum_{j\in J}t_j\,\scan{\ecG_j}\prod_i (\scan{\ecD_i}\,\scan{\ecF_i})^{b_{i,j}}=0,
\end{equation}
et on pose 
\begin{equation}
Z_K(t)
\eqdef
\sum_{\becE\in \diveffc^{I\cup J}}
\mu_X(\becE)
\sum_{
\substack{
y\in\Pic(X)^{\vee}\cap \ceff(X)^{\vee}
\\
\acc{y}{\scG_j}\geq \deg(\ecG_j),\quad j\in J
\\
\acc{y}{\scF_i}\geq \deg(\ecF_i),\quad i\in I
\\
\becD\in \diveffc^I\\
\deg(\ecD_i)=\acc{y}{\scF_i}-\deg(\ecF_i),\quad i\in I 
}
}
\fonc_K(\becD,\becE)
\,\,
t^{\acc{y}{-\can{X}}}.
\end{equation}
On a donc d'apr\`es \eqref{eq:form:relevement} la relation
\begin{equation}
Z_{X,-\can{X},X_0}(t)
=
\sum_{K\subset J}
(-1)^{\card{J}-\card{K}} Z_K(t).
\end{equation}

\begin{prop}\label{prop:maj:fonc:j}
Soit $\becE\in \diveffc^{I\cup J}$, 
 $\becD\in \diveffc^I$ et 
$y\in\Pic(X)^{\vee}\cap \ceff(X)^{\vee}$
tels que pour $i\in I$ on ait $\deg(\ecD_i)=\acc{y}{\scF_i}-\deg(\ecF_i)$.
\begin{enumerate}
\item\label{item:prop:maj:fonc:j0}
Pour tout $j_0\in J$, on  a
\begin{equation}\label{eq:maj:fonc:J:j0}
\log_q \fonc_{J\setminus \{j_0\}}(\becD,\becE)\leq 1+\left(1-\frac{1}{\card{J}-1}\right)\sum_{j\neq j_0}(\acc{y}{\scG_j}-\deg(\ecG_j))
\end{equation}
\item\label{item:prop:maj:fonc:j}
La quantit\'e $\log_q \fonc_{J}(\becD,\becE)$ est major\'ee soit par
\begin{equation}
\card{J}-1+\acc{y}{-\Dtot +\sum_{j\in J}\scG_j}-\sum_{j\in J} \deg(\ecG_j)
+
\deg\left(\Infu{j\in J}(\sumu{i\in I} b_{i,j} (\ecF_i+\ecD_i)+\ecG_j)\right)
\end{equation}
soit par
\begin{equation}
1+\left(1-\frac{1}{\card{J}-1}\right)\sum_{j\in J}(\acc{y}{\scG_j}-\deg(\ecG_j)).
\end{equation}
\item\label{item:prop:maj:eq:j}
On suppose que pour une certaine num\'erotation de $J$, on a 
\begin{equation}
\forall 1\leq j\leq \card{J}-1,\quad
\acc{y}{\scG_j+\scG_{j+1}-\Dtot}
\geq 
\deg(\ecG_j)+\deg(\ecG_{j+1})
+2\,\gc-1.
\end{equation}
Alors on a 
\begin{multline}
\log_q\fonc_J(\becD,\becE)=
\\
(\card{J}-1)\,(1-\gc)+\acc{y}{-\Dtot
+\sum_{j\in J}
  \scG_j}-\sum_{j\in J} \deg(\ecG_j)
+\deg\left(\Infu{j\in J}(\sumu{i\in I} b_{i,j} (\ecF_i+\ecD_i)+\ecG_j)\right).
\end{multline}
\end{enumerate}
\end{prop}
\begin{proof}
On applique la proposition \ref{prop:compt} avec 
\begin{equation}
\ecH'_j=-\ecG_j+\sum_{i\in I} a_{i,j}\,(\ecD_i+\ecF_i),
\quad\quad
\ecH_j=\ecG_j+\sum_{i\in I} b_{i,j}\,(\ecD_i+\ecF_i)
\end{equation}
\begin{equation}
\text{et}\quad\quad
s_j=\scan{\ecG_j}\prod_{i\in I} (\scan{\ecD_j}\,\scan{\ecF_i})^{b_{i,j}},
\end{equation}
en notant qu'on a  
$
\deg(\ecH)=\acc{y}{\Dtot}
$
et, pour $j\in J$, 
$
\deg(\ecH'_j)=\acc{y}{\scG_j}-\deg(\ecG_j)
$.
\end{proof}
\begin{defi}
Un \'el\'ement $(\bbf,\bg)\in (\bR_{\geq 0})^{I\cup J}$ est dit
$\mu_X$-convergent si la s\'erie
\begin{equation}\label{eq:sum:mubecE}
\sum_{\becE\in \diveffc^{I\cup J}}
\abs{\mu_X(\becE)}
q^{-\sum f_i\deg(\ecF_i)-\sum g_j \deg(\ecG_j)}
\end{equation}
est convergente.
\end{defi}
\begin{rem}\label{rem:inter}
On d\'eduit du dernier point de la proposition \ref{prop:mu} 
et de l'\'ecriture de la s\'erie \eqref{eq:sum:mubecE} sous forme de
produit eul\'erien 
les propri\'et\'es suivantes~:
\begin{enumerate}
\item
pour tout $\eps>0$, $(\frac{1}{2}+\eps,\dots,\frac{1}{2}+\eps)$ est
$\mu_X$-convergent~;
\item
on suppose l'intersection des diviseurs $\{\scG_j\}_{j\in J}$ non vide~;
alors pour tout $\eps>0$, $((1)_{i\in I},(\eps)_{j\in J})$ est $\mu_X$-convergent.
\end{enumerate}
\end{rem}
\begin{lemme}\label{lm:muxconv}
Soit $\scD\in \Pic(X)$ et  $\bgamma\in\bR^J$. On suppose qu'il existe
$(\bbf,\bg)\in (\bR_{\geq 0})^{I\cup J}$
v\'erifiant
\begin{enumerate}
\item
$((1+f_i),(g_j+\gamma_j))$
est $\mu_X$-convergent~;
\item
$-\can{X}-\scD-\sumu{i\in I}(1+f_i)\scF_i-\sumu{j\in J} g_j \,\scG_j\in\ceff(X)\setminus \{0\}$.
\end{enumerate}
Alors la s\'erie
\begin{equation}
Z(\scD,\bgamma,t)\eqdef \sum_{\becE\in \diveffc^{I\cup J}}
\abs{\mu_X(\becE)}
\sum_{
\substack{
y\in\Pic(X)^{\vee}\cap \ceff(X)^{\vee}
\\
\acc{y}{\scF_i}\geq \deg(\ecF_i),\quad i\in I
\\
\acc{y}{\scG_j}\geq \deg(\ecG_j),\quad j\in J
\\
\becD\in \diveffc^I
\\
\deg(\ecD_i)=\acc{y}{\scF_i}-\deg(\ecF_i),\quad i\in I
}
}
q^{\acc{y}{\scD}-\sumu{j\in J}\gamma_j\,\deg(\ecG_j)}
\,
t^{\acc{y}{-\can{X}}}
\end{equation}
est $q^{-1}$-contr\^ol\'ee \`a l'ordre
$\rg(\Pic(X))-1$.
\end{lemme}
\begin{proof}
D'apr\`es \eqref{eq:maj:D}, elle est en effet major\'ee \`a une constante multiplicative pr\`es par la s\'erie
\begin{multline}
\sum_{\becE\in \diveffc^{I\cup J}}
\abs{\mu_X(\becE)}
\\
\times\!\!\!\!\!\!\!\!
\sum_{
\substack{
y\in\Pic(X)^{\vee}\cap \ceff(X)^{\vee}
\\
\acc{y}{\scF_i}\geq \deg(\ecF_i),\quad i\in I
\\
\acc{y}{\scG_j}\geq \deg(\ecG_j),\quad j\in J
}
}
\!\!\!\!\!\!\!\!
q^{\acc{y}{\scD}
\sumu{j\in J}
g_j\,
(\acc{y}{\scG_j}-\deg(\ecG_j))
-\gamma_j\,\deg(\ecG_j)
+
\sumu{i\in I}(1+f_i)(\acc{y}{\scF_i}-\deg(\ecF_i))
}
t^{\acc{y}{-\can{X}}}
\end{multline}
elle-m\^eme major\'ee par
\begin{multline}
\sum_{\becE\in \diveffc^{I\cup J}}
\abs{\mu_X(\becE)}
q^{-\sumu{i\in I}(1+f_i)\,\deg(\ecF_i)-\sumu{j\in J}(g_j+\gamma_j)\deg(\ecG_j)}
\\
\times
\sum_{
\substack{
y\in\Pic(X)^{\vee}\cap \ceff(X)^{\vee}
}
}
q^{\acc{y}{\scD+\sumu{i\in I}(1+f_i){\scF_i}+\sumu{j\in J}g_j{\scG_j}}}
\,
t^{\acc{y}{-\can{X}}}.
\end{multline}
D'apr\`es les hypoth\`eses et le lemme \ref{lm:prelim} cette derni\`ere
s\'erie est $q^{-1}$-contr\^ol\'ee \`a l'ordre
$\rg(\Pic(X))-1$.
\end{proof}
\begin{rem}\label{rem:inter:bis}
Supposons que l'intersection des diviseurs $\{\scG_j\}_{j\in J}$ est non vide.
Soit $\scD\in \Pic(X)$ et $\bgamma\in \bR_{\geq 0}^J$. 
Notons $J_0\eqdef\{j\in J,\,\gamma_j=0\}$. 

Supposons en outre 
que
pour tout $\eps>0$ assez petit on a
\beq
\sum_{j\in J_0\setminus J_1}(1-\eps)\,\scG_j
+
\sum_{j\notin  J_0}(1-\gamma_j)\scG_j-\Dtot-\scD\in \ceff(X)\setminus \{0\}.
\eeq
La remarque \ref{rem:inter}, la formule d'adjonction \eqref{eq:adj} et le lemme pr\'ec\'edent
montrent alors que la s\'erie $Z(\scD+\sum \gamma_j \,\scG_j,\bgamma,t)$ 
est $q^{-1}$-contr\^ol\'ee \`a l'ordre
$\rg(\Pic(X))-1$.
\end{rem}

\subsection{Le terme $Z_J$}

Pour $\becD\in \diveffc^I$, $\becE\in \diveffc^{I\cup J}$ et $y\in\Pic(X)^{\vee}$
on d\'efinit $\wfonc(\becD,\becE,y)$  par 
\begin{multline}
\log_q(\wfonc(\becD,\becE,y))
\\
=
\sum_j \acc{y}{\scG_j}
-\acc{y}{\Dtot}
-\sum_j \deg(\ecG_j)
+\deg\left[\Inf_j(\sum_i b_{i,j}(\ecD_i+\ecF_i)+\ecG_j)\right].
\end{multline}
Posons
\begin{equation}
\Zprgen(t)
\eqdef
\sum_{\becE\in \diveffc^{I\cup J}}
\mu_X(\becE)
\!\!\!\!\!\!\!\!\!\!\!\!
\sum_{
\substack{
y\in\Pic(X)^{\vee}\cap \ceff(X)^{\vee}
\\
\acc{y}{\scG_j}\geq \deg(\ecG_j),\quad j\in J
\\
\acc{y}{\scF_i}\geq \deg(\ecF_i),\quad i\in I
\\
\becD\in \diveffc^I
\\
\deg(\ecD_i)=\acc{y}{\scF_i}-\deg(\ecF_i),\quad i\in I 
}
}
\!\!\!\!\!\!\!\!\!\!\!\!
q^{(\card{J}-1)(1-\gc)}
\,
\wfonc(\becD,\becE,y)
\,
t^{\acc{y}{-\can{X}}}
\end{equation} 
et, pour $j_0,j_1\in J$,
\begin{equation}
\Zerr{1,j_0,j_1}(t)
\eqdef
\sum_{\becE\in \diveffc^{I\cup J}}
\abs{\mu_X(\becE)}
\!\!\!\!\!\!\!\!\!\!\!\!\!\!\!\!\!\!\!\!\!\!\!\!
\sum_{
\substack{
y\in\Pic(X)^{\vee}\cap \ceff(X)^{\vee}
\\
\acc{y}{\scG_j}\geq \deg(\ecG_j),\quad j\in J
\\
\acc{y}{\scF_i}\geq \deg(\ecF_i),\quad i\in I
\\
\acc{y}{\scG_{j_0}+\scG_{j_1}-\Dtot}
\leq 
\deg(\ecG_{j_0})+
\deg(\ecG_{j_1})+
2\,\gc-2
\\
\becD\in \diveffc^I\\
\deg(\ecD_i)=\acc{y}{\scF_i}-\deg(\ecF_i),\quad i\in I 
}
}
\!\!\!\!\!\!\!\!\!\!\!\!\!\!\!\!\!\!\!\!\!\!\!\!
q^{(\card{J}-1)(1-\gc)}\wfonc(\becD,\becE,y)
\,
t^{\acc{y}{-\can{X}}}
\end{equation}
\begin{equation}
\text{et}\quad
\Zerr{2,j_0,j_1}(t)
\eqdef
\sum_{\becE\in \diveffc^{I\cup J}}
\abs{\mu_X(\becE)}
\!\!\!\!\!\!\!\!\!\!\!\!\!\!\!\!\!\!\!\!\!\!\!\!
\sum_{
\substack{
y\in\Pic(X)^{\vee}\cap \ceff(X)^{\vee}
\\
\acc{y}{\scG_j}\geq \deg(\ecG_j),\quad j\in J
\\
\acc{y}{\scF_i}\geq \deg(\ecF_i),\quad i\in I
\\
\acc{y}{\scG_{j_0}+\scG_{j_1}-\Dtot}
\leq 
\deg(\ecG_{j_0})+
\deg(\ecG_{j_1})+
2\,\gc-2
\\
\becD\in \diveffc^I\\
\deg(\ecD_i)=\acc{y}{\scF_i}-\deg(\ecF_i),\quad i\in I 
}
}
\!\!\!\!\!\!\!\!\!\!\!\!\!\!\!\!\!\!\!\!\!\!\!\!
\fonc_{J}(\becD,\becE)
\,
t^{\acc{y}{-\can{X}}}.
\end{equation}
Appliquant la proposition \ref{prop:maj:fonc:j}, on obtient le lemme suivant.
\begin{lemme}
Pour toute num\'erotation de $J$,
la s\'erie
$
Z_{J}(t)-\Zprgen(t)
$
est major\'ee par 
\begin{equation}
\sum_{1\leq j\leq \card{J}-1}\Zerr{1,j,j+1}(t)+\Zerr{2,j,j+1}(t).
\end{equation}
\end{lemme}
\subsection{Le terme $\Zprgen$}
Pour $\bd\in \bN^I$ et $\becE\in \diveffc^{I\cup J}$, 
posons
\begin{equation}
\fonc(\bd,\becE)
\eqdef
\sum_{
\substack{
\becD\in \diveffc^I
\\
\deg(\becD)=\bd
}}
q^{\deg(\Inf_j(\ecG_j+\sum_i b_{i,j}(\ecD_i+\ecF_i)))}.
\end{equation}
Compte tenu de la relation \eqref{eq:adj} et de l'\'egalit\'e
$\dim(X)=\card{J}-1$, 
on a donc
\begin{multline}\label{eq:expr:zprgen}
q^{\,(\gc-1)\dim(X)}\Zprgen(t)=
\\
\sum_{\becE\in \diveffc^{I\cup J}}
\mu_X(\becE)
\!\!\!\!\!\!\!\!\!\!\!\!
\sum_{
\substack{
y\in\Pic(X)^{\vee}\cap \ceff(X)^{\vee}
\\
\acc{y}{\scG_j}\geq \deg(\ecG_j),\quad j\in J
\\
\acc{y}{\scF_i}\geq \deg(\ecF_i),\quad i\in I
}
}
\!\!\!\!\!\!\!\!\!\!\!\!
q^{-\sum_j \deg(\ecG_j)-\sum_{i}\acc{y}{\scF_i}}\,
\fonc(\acc{y_i}{\scF_i}-\deg(\ecF_i),\becE)
\,
(q\,t)^{\acc{y}{-\can{X}}}
\end{multline}

\subsubsection{Estimation de $\fonc(\bd,\becE)$}\label{subsubsec:estim:fonc1}
Pour $\be=(\bbf,\bg)\in \bN^{I\cup J}$, on pose
\begin{equation}\label{eq:def:Frhoalphabeta}
F_{\rho,\be}(t)
\eqdef
\sum_{\bd\in \bN^I}
\rho^{\Inf_j
\left(
g_j+
\sum_i
  b_{i,j}(d_i+f_i)
\right)
}
\bt^{\,\bd}\in k[[\rho,(t_i)_{i\in I}]],
\end{equation}
\begin{equation}\label{eq:def:wtFrhoalphabeta}
\text{et}\quad\wt{F}_{\rho,\be}\eqdef\left(\prod_{i\in
    I}1-\,t_i\right)\,F_{\rho,\be}
\eqdef 
\sum_{\bd\in \bN^I}P_{\be,\bd}(\rho)\,\bt^{\bd},
\end{equation}
o\`u $P_{\be,\bd}(\rho)$ est un polyn\^ome en $\rho$ \`a coefficients entiers.
Il d\'ecoule aussit\^ot de 
\eqref{eq:def:Frhoalphabeta}
et \eqref{eq:def:wtFrhoalphabeta} qu'on a 
$P_{\be,0}
=\rho^{\Inf_j\left(\sum_i  b_{i,j}f_i+g_j\right)}$ 
et que pour tout $\bd$, le polyn\^ome
$P_{\be,\bd}(\rho)$ a au
plus $\card{I}$ coefficients non nuls, qui sont tous major\'es en valeur
absolue par $\card{I}$.
\begin{hyp}\label{hyp:conv:sergen}
\begin{enumerate}
\item
Pour tout  $\eta>0$ assez petit, et
pour tout $\bd\neq 0$, on a 
\begin{equation}\label{eq:maj:deg:P0d}
(1-\eta)\,\abs{\bd}\geq 1+\eta+\deg(P_{0,\bd})\quad ;
\end{equation}
\item
Pour $\be\in \{0,1\}^{I\cup J}$, on a une \'ecriture
\beq
\wt{F}_{\rho,\be}(\bt)=
\left(1+\sum_{\bd\neq 0}Q_{\be,\bd}(\rho)\bt^{\bd}\right)\,R_{\be}(\rho,\bt)
\eeq
o\`u $Q_{\be,\bd}(\rho)$ est un polyn\^ome en $\rho$ et $R_{\be}$ un
polyn\^ome en $\rho$ et $\bt$~;
en outre, pour tout  $\eta>0$ assez petit, et
pour tout $\bd\neq 0$, on a
\begin{equation}\label{eq:maj:deg:P0d:bis}
(1-\eta)\,\abs{\bd}\geq 1+\eta+\deg(Q_{\be,\bd}).
\end{equation}
\end{enumerate}
\end{hyp}
\begin{rem}\label{rem:hyp:conv:sergen}
La majoration \eqref{eq:maj:deg:P0d} est v\'erifi\'e d\`es qu'on a une \'ecriture
$
\wt{F}_{\rho,0}(\bt)=\prod_{\alpha\in A}G_{\alpha,\rho}(\bt)^{\eps_{\alpha}}
$
avec $A$ fini, $\eps_{\alpha}\in \{-1,1\}$ et $G_{\alpha,\rho}(\bt)$ est un polyn\^ome
en $\bt$ de terme constant $1$ et pour lequel le coefficient de
$\bt^{\bd}$, pour $\bd\neq 0$, est un polyn\^ome en 
$\rho$ de degr\'e strictement inf\'erieur \`a $\abs{\bd}-1$.
\end{rem}
\begin{rem}\label{rem:courbe:eta}
Si l'hypoth\`ese \ref{hyp:conv:sergen} vaut, on v\'erifie aussit\^ot que
pour tout $\eta>0$ assez petit, pour tout $v\in \courbe^{(0)}$ et tout  $\bt\in \bC^I$
v\'erifiant $\norm{\bt}\leq q^{-1+\eta}$ on a 
\beq
\abs{\wt{F}_{q_v,0}(\bt)-1}=\ecO_{\eta}\left(q_v^{-1-\frac{1}{2}\eta}\right).
\eeq
Il existe donc en particulier un ensemble $\courbe^{(0)}_{\eta}$ fini tel qu'on ait
\begin{equation}
\forall v\notin \courbe^{(0)}_{\eta},\,\forall \bt\in \bC^I,\,
\left(\norm{\bt}\leq q^{-1+\eta}
  \imply \abs{\wt{F}_{q_v,0}(\bt)}\geq \frac{1}{2}\right).
\end{equation}
\end{rem}
Pour $\becE\in \diveffc^{I\cup J}$,
on consid\`ere \`a pr\'esent la s\'erie g\'en\'eratrice
\begin{equation}
Z_{\becE}(\bt)
\eqdef\sum_{\bd\in \bN^{I}}\fonc(\bd,\becE)\bt^{\,\bd}
=\sum_{\becD\in \diveffc^I}
q^{\deg(\Inf_j\left(\ecG_j+\sum_i  b_{i,j}(\ecD_i+\ecF_i)\right))} 
\bt^{\,\deg(\becD)}
\end{equation}
qui s'exprime donc comme le produit eul\'erien
$
\prod_{v\in \courbe^{(0)}} \frac{\wt{F}_{q_v,v(\becE)}(\bt^{f_v})}{\produ{i \in I}(1-t_i^{f_v})}
$.
On a la relation
\begin{align}
\left(\prod_{i\in I} 1-q\,t_i\right)\,Z_{\becE}(\bt)
=
\left(\prod_{i\in I} (1-q\,t_i)\,Z_{\courbe}(t_i)\right)
\prod_{v\in \courbe^{(0)}}\wt{F}_{q_v,v(\becE)}(\bt^{f_v}),
\end{align}
o\`u $Z_{\courbe}(t)\eqdef\sum_{D\in \diveffc}t^{\deg(D)}$ est la
fonction z\^eta de Hasse-Weil de $\courbe$.
Si on a $v(\becE)\in \{0,1\}^{I\cup J}$ et si l'hypoth\`ese \ref{hyp:conv:sergen}
est v\'erifi\'ee, la s\'erie
$\left(\prod_{i\in I} 1-q\,t_i\right)\,Z_{\becE}(\bt)$
converge donc absolument sur un polydisque de rayon $(q^{-1+\eta},\dots,q^{-1+\eta})$ pour $\eta>0$ assez petit.
Notons $\courbe^{(0)}_{\becE}$ l'ensemble (fini) des $v\in \courbe^{(0)}$
tels que $v(\becE)\neq 0$.
Si $\eta>0$ est assez petit et si $\courbe^{(0)}_{\eta}$ est
l'ensemble 
introduit \`a la remarque \ref{rem:courbe:eta}
on peut alors \'ecrire
\begin{multline}
\left(\prod_{i\in I} 1-q\,t_i\right)\,Z_{\becE}(\bt)
\\
=
\left(\prod_i (1-q\,t_i)\,Z_{\courbe}(t_i)\right)
\prod_{v\in \courbe^{(0)}\setminus \courbe^{(0)}_{\eta}}\wt{F}_{q_v,0}(\bt)
\prod_{v\in \courbe^{(0)}_{\eta}}\wt{F}_{q_v,v(\becE)}(\bt)
\prod_{v\in \courbe^{(0)}_{\becE}\setminus \courbe^{(0)}_{\eta}}\frac{\wt{F}_{q_v,v(\becE)}(\bt)}{\wt{F}_{q_v,0}(\bt)}.
\label{eq:ecr}
\end{multline}
\begin{nota}
Si l'hypoth\`ese \ref{hyp:conv:sergen} est v\'erifi\'ee, pour tout $\becE\in
\diveffc^{I\cup J}$ tel que $v(\becE)\in \{0,1\}^{I\cup J}$, 
on note 
\begin{equation}
\cprinc(\becE)\eqdef
\left(\frac{\hc\,q^{1-\gc}}{q-1}\right)^{\card{I}}\,
\prod_{v\in \courbe^{(0)}}\wt{F}_{q_v,v(\becE)}(q_v^{-1})
\end{equation}
et pour tout $\eta>0$ assez petit
\begin{equation}
c_{\eta,\becE}\eqdef\norm{\prod_{v\in \courbe^{(0)}_{\becE}\setminus \courbe^{(0)}_{\eta}}\frac{\wt{F}_{q_v,v(\becE)}(\bt)}{\wt{F}_{q_v,0}(\bt)}}_{q^{-1+\eta}}.
\end{equation}
\end{nota}
De ce qui pr\'ec\`ede et du lemme \ref{lm:estim} d\'ecoule alors aussit\^ot la
proposition suivante.
\begin{prop}\label{prop:maj:n1}
On suppose l'hypoth\`ese \ref{hyp:conv:sergen} v\'erifi\'ee.
Pour tout $\eta>0$ assez petit, il existe une constante $C_{\eta}$ 
telle qu'on ait, pour tout $\bd$ et tout $\becE\in
\diveffc^{I\cup J}$ v\'erifiant $v(\becE)\in \{0,1\}^{I\cup J}$, la majoration
\begin{equation}
\abs{\fonc(\bd,\becE)
-\cprinc(\becE)
\,q^{\sum d_j}}
\leq
C_{\eta}.
c_{\eta,\becE}
\sum_{i\in I}
q^{(1-\eta)d_i+\sumu{i'\neq i} d_{i'}}.
\end{equation}
\end{prop}
\begin{hyp}\label{hyp:conv:cEeps}
L'hypoth\`ese \ref{hyp:conv:sergen} est v\'erifi\'ee et 
pour tout  $\eta>0$ assez petit la s\'erie 
\begin{equation}\label{eq:ser:ceta:becE}
\sum_{\becE\in \diveffc^{I\cup J}}
\abs{\mu_X(\becE)}
c_{\eta,\becE}\,
q^{\,-\sumu{i\in I} (1-\eta)\deg(\ecF_i)
-\sumu{j\in J} \deg(\ecG_j)}
\end{equation}
converge.
\end{hyp}
\begin{rem}\label{rem:hyp:conv:cEeps}
Supposons l'hypoth\`ese \ref{hyp:conv:sergen} v\'erifi\'ee. 
Pour $\be\in \{0,1\}^{I\cup J}\setminus \{(0,\dots,0)\}$,
\'ecrivons 
\beq
R_{\be}(\rho,\bt)=\sum_{\bd} R_{\be,\bd}(\rho)\,\bt^{\bd}
\eeq
et notons $C_{\be}$ une constante v\'erifiant
\beq
\forall \bd\in \bN^I,\quad  \deg(R_{\be,\bd})\leq \abs{\bd}+C_{\be}.
\eeq
Il existe alors une constante $M$ telle que pour $\eta>0$ assez petit,
le degr\'e du polyn\^ome $P_{\be,\bd}(\rho)$ est major\'e par
$\abs{\bd}(1-\eta)+C_{\be}+\eta\,M$.
Ceci entra\^\i ne l'existence d'une constante $C>0$ telle qu'on ait pour tout $\becE$
et tout $\eta$ assez petit la majoration
$
c_{\eta,\becE}
\leq \prod_{v\in \courbe^{(0)}_{\becE}} C\,q_v^{C_{v(\becE)}+2\,\eta\,M}.
$
Ainsi pour $\eta>0$ assez petit la s\'erie \eqref{eq:ser:ceta:becE} est major\'ee par le produit
eul\'erien
\beq
\prod_{v\in \courbe^{(0)}}
1+
C\,\sum_{\be\in \{0,1\}^{I\cup J},\,\be\neq 0}
\abs{\mu_X^0(\be)}
q_v^{C_{\be}+2\,\eta\,M-\sum_i (1-\eta)\,f_i-\sum_jg_j}
\eeq
et l'hypoth\`ese \ref{hyp:conv:cEeps} sera satisfaite d\`es que la
propri\'et\'e suivante est v\'erifi\'ee~:
\beq
\forall \be=(\bbf,\bg)\in \{0,1\}^{I\cup J}\setminus \{(0,\dots,0)\},\quad
\left(\mu_X^0(\be)\neq 0 \imply C_{\be}-\sum_i f_i-\sum_jg_j<-1\right).
\eeq
\end{rem}
\subsubsection{Le terme $\Zprgen$}
Si l'hypoth\`ese \ref{hyp:conv:sergen} est v\'erifi\'ee, 
posons 
\begin{multline}
\Zprprgen(t)
\eqdef
q^{\,\dim(X)\,(1-\gc)}
\sum_{\becE\in \diveffc^{I\cup J}}
\mu_X(\becE)
\cprinc(\becE)
q^{-\sum_i \deg(\ecF_i)
-\sum_j \deg(\ecG_j)}
\\
\times
\sum_{
\substack{
y\in\Pic(X)^{\vee}\cap \ceff(X)^{\vee}
\\
\acc{y}{\scG_j}\geq \deg(\ecG_j),\quad j\in J
\\
\acc{y}{\scF_i}\geq \deg(\ecF_i),\quad i\in I
}
}
(q\,t)^{\acc{y}{-\can{X}}}
\end{multline}
\begin{lemme}
On suppose l'hypoth\`ese \ref{hyp:conv:cEeps} v\'erifi\'ee.
La s\'erie
$
\Zprgen-\Zprprgen
$
est alors $q^{-1}$-contr\^ol\'ee
\`a l'ordre $\rg(\Pic(X))-1$.
\end{lemme}
\begin{proof}
Posons pour $i_0\in I$ et $\eta>0$ assez petit
\begin{multline}
\Zpr{\eta,i_0}(t)
\eqdef
C_{\eta}
\sum_{\becE\in \diveffc^{I\cup J}}
\abs{\mu_X(\becE)}
c_{\eta,\becE}
q^{\eta\,\deg(\ecF_{i_0}) -\sum_i \deg(\ecF_i)
-\sum_j \deg(\ecG_j)}
\\
\times
\sum_{
\substack{
y\in\Pic(X)^{\vee}\cap \ceff(X)^{\vee}
\\
}
}
q^{-\eta\acc{y}{\scF_{i_0}}}
(q\,t)^{\acc{y}{-\can{X}}}
\end{multline}
D'apr\`es le lemme \ref{lm:prelim} et l'hypoth\`ese \ref{hyp:conv:cEeps}, 
pour $\eta$ assez petit,
la s\'erie $\Zpr{\eta,i}$ est $q^{-1}$-contr\^ol\'ee
\`a l'ordre $\rg(\Pic(X))-1$ pour tout $i\in I$.
Or d'apr\`es la proposition \ref{prop:maj:n1}, la s\'erie
$
\Zprgen-\Zprprgen
$
est major\'ee par 
$\sum_{i\in I} \Zpr{\eta,i}$.
\end{proof}

\subsubsection{Le terme $\Zprprgen$}
\begin{hyp}\label{hyp:rel:term:princ}
On a pour tout $v\in \courbe^{(0)}$ la relation
\begin{multline}\label{eq:hyp:rel:term:princ}
\sum_{(\bbf,\bg)\in \{0,1\}^{I\cup J}}
\mu^0_X(\bbf,\bg)
q_v^{-\sum f_i-\sum g_j}
\wt{F}_{q_v,\bbf,\bg}(q_v^{-1})
=
(1-q_v^{-1})^{\rg(\TNS(X))}\frac{\card{X(\kappa_v)}}{q_v^{\dim(X)}}.
\end{multline}
\end{hyp}
\begin{rem}
On notera l'analogie avec la formule de \cite[Lemme 1.25]{Bou:compt}.
\end{rem}
\begin{rem}\label{rem:eff}
Au moins dans le cas o\`u pour tout $i\in I$ l'un au plus des
$\{b_{i,j}\}_{j\in J}$ est non nul, il est facile de voir que les
$F_{\rho,\be}(t)$ sont des fractions rationnelles effectivement
calculables, ce qui peut alors permettre, pour un exemple donn\'e dont
on conna\^\i t l'anneau de Cox et les relations d'incidence des diviseurs
$\scF_i$ et $\scG_j$, 
de confier aux soins d'un logiciel de
calcul formel la v\'erification des  hypoth\`eses \ref{hyp:conv:sergen}, \ref{hyp:conv:cEeps} et 
\ref{hyp:rel:term:princ}.
\end{rem}
\begin{prop}
On suppose les hypoth\`eses \ref{hyp:conv:sergen} et
\ref{hyp:rel:term:princ} satisfaites. 
Alors 
la s\'erie 
\begin{equation}
\Zprprgen(t)
-
\gamma(X)\!\!
\sum_{
y\in\Pic(X)^{\vee}\cap \ceff(X)^{\vee}
}
(q\,t)^{\acc{y}{-\can{X}}}
\end{equation}
est $q^{-1}$-contr\^ol\'ee \`a l'ordre $\rg(\Pic(X))-1$.
\end{prop}
\begin{proof}
Rappelons qu'on a
\begin{equation}
\gamma(X)
=
\left(\frac{\hc\,q^{(1-\gc)}}{q-1}\right)^{\rg(\Pic(X))}\:
q^{(1-\gc)\,\dim(X)}\,
\prod_{v\in\courbe^{(0)}}
(1-q_v^{-1})^{\rg(\Pic(X))}\,\frac{\card{X(\kappa_{v})}}{q_v^{\,\dim(X)}}.
\end{equation}
D'apr\`es le lemme \ref{lm:prelim} et la d\'efinition de $\Zprprgen$, la s\'erie
\begin{multline}
\Zprprgen(t)-q^{\,\dim(X)\,(1-\gc)}
\sum_{\becE\in \diveffc^{I\cup J}}
\mu_X(\becE)\,
\cprinc(\becE)\,
q^{-\sum_i \deg(\ecF_i)
-\sum_j \deg(\ecG_j)}
\\
\times
\sum_{
y\in\Pic(X)^{\vee}\cap \ceff(X)^{\vee}
}
(q\,t)^{\acc{y}{-\can{X}}}
\end{multline}
est $q^{-1}$-contr\^ol\'ee \`a l'ordre 
$\rg(\Pic(X))-1$. Or on a 
\begin{multline}
q^{\,\dim(X)\,(1-\gc)}
\sum_{\becE\in \diveffc^{I\cup J}}
\mu_X(\becE)\,
\cprinc(\becE)\,
q^{-\sum_i \deg(\ecF_i)
-\sum_j \deg(\ecG_j)}
\\
\shoveleft{=
\left(\frac{\hc\,q^{1-\gc}}{q-1}\right)^{\rg(\Pic(X))}
q^{\dim(X)(1-\gc)}}
\\
\times
\sum_{\becE\in \diveffc^{I\cup J}}
\mu_X(\becE)
\left[\prod_v\wt{F}_{q_v,v(\becF),v(\becG)}(q_v^{-1})\right]
q^{-\sum_i \deg(\ecF_i)
-\sum_j \deg(\ecG_j)}
\end{multline}
et
\begin{multline}
\sum_{\becE\in \diveffc^{I\cup J}}
\mu_X(\becE)
\left[\prod_v\wt{F}_{q_v,v(\becF),v(\becG)}(q_v^{-1})\right]
q^{-\sum_i \deg(\ecF_i)
-\sum_j \deg(\ecG_j)}
\\
=
\prod_v
\sum_{(\bbf,\bg)\in \{0,1\}^{I\cup J}}
\mu^0_X(\bbf,\bg)
q_v^{-\sum_i f_i-\sum_j g_j}
\wt{F}_{q_v,\bbf,\bg}(q_v^{-1})
\end{multline}
L'hypoth\`ese \ref{hyp:rel:term:princ} donne alors le r\'esultat.
\end{proof}
\subsection{Les termes $\Zerr{1,j_0,j_1}$ et $\Zerr{2,j_0,j_1}$}
\subsubsection{Les termes $\Zerr{1,j_0,j_1}$}
Soit $j_0,j_1\in J$.
On a pour
$\Zerr{1,j_0,j_1}$ une expression semblable \`a \eqref{eq:expr:zprgen}, avec $\abs{\mu_X(\becE)}$ en lieu et place de
$\mu_X(\becE)$ et le deuxi\`eme domaine de sommation restreint aux $y$ v\'erifiant
\begin{equation}
\acc{y}{\scG_{j_0}+\scG_{j_1}-\Dtot}
\leq 
\deg(\ecG_{j_0})+
\deg(\ecG_{j_1})+
2\,\gc-2.
\end{equation}

Par un raisonnement analogue \`a celui effectu\'e pour $\Zprprgen$,
on montre que si l'hypoth\`ese \ref{hyp:conv:cEeps} est satisfaite la s\'erie 
\begin{multline}
\Zerr{1,j_0,j_1}(t)
\\
-
q^{\,\dim(X)\,(1-\gc)}
\!\!\!\!\!\!\!\!
\sum_{\becE\in \diveffc^{I\cup J}}
\!\!\!\!\!\!\!\!
\abs{\mu_X(\becE)}
\cprinc(\becE)\,
q^{-\sum_i \deg(\ecF_i)
-\sum_j \deg(\ecG_j)}
\!\!\!\!\!\!\!\!\!\!\!\!\!\!\!\!\!\!
\!\!\!\!\!\!\!\!\!\!\!\!\!\!\!\!\!\!
\sum_{
\substack{
y\in\Pic(X)^{\vee}\cap \ceff(X)^{\vee}
\\
\acc{y}{\scG_{j_0}+\scG_{j_1}-\Dtot}
\leq 
\deg(\ecG_{j_0})+
\deg(\ecG_{j_1})+
2\,\gc-2
}
}
\!\!\!\!\!\!\!\!\!\!\!\!\!\!\!\!\!\!
\!\!\!\!\!\!\!\!\!\!\!\!\!\!\!\!\!\!
(q\,t)^{\acc{y}{-\can{X}}}
\end{multline}
est $q^{-1}$-contr\^ol\'ee \`a l'ordre $\rg(\Pic(X))-1$.
En utilisant le lemme \ref{lm:prelim}, on obtient aussit\^ot ce qui suit.
\begin{prop}\label{prop:need:hyp:I}
Soit $j_0,j_1\in J$ tels qu'on ait 
\begin{equation}
\scG_{j_0}+\scG_{j_1}-\Dtot\in \ceff(X)\setminus \{0\}.
\end{equation}
On suppose que l'hypoth\`ese \ref{hyp:conv:cEeps} est satisfaite.
Alors $\Zerr{1,j_0,j_1}$ est $q^{-1}$-contr\^ol\'ee \`a l'ordre $\rg(\Pic(X))-1$
\end{prop}

\subsubsection{Les termes $\Zerr{2,j_0,j_1}$}
Soit $j_0,j_1\in J$.
En appliquant la proposition \ref{prop:maj:fonc:j}, on voit que $\Zerr{2,j_0,j_1}$
est major\'ee par la somme des deux s\'eries
\begin{equation}\label{eq:zerr2j0j1:1}
\sum_{\becE\in \diveffc^{I\cup J}}
\!\!\!\!\!\!\!\!\!\!\!\!
\abs{\mu_X(\becE)}
\sum_{
\substack{
y\in\Pic(X)^{\vee}\cap \ceff(X)^{\vee}
\\
\acc{y}{\scG_j}\geq \deg(\ecG_j),\quad j\in J
\\
\acc{y}{\scF_i}\geq \deg(\ecF_i),\quad i\in I
\\
\becD\in \diveffc^I\\
\deg(\ecD_i)=\acc{y}{\scF_i}-\deg(\ecF_i),\quad i\in I 
}
}
\!\!\!\!\!\!\!\!\!\!\!\!
q^{\card{J}-2+(1-\frac{1}{\card{J}-1})\left(\acc{y}{\sum_{j} \scG_j}-\sum_{j}\deg(\ecG_j)\right)}
t^{\acc{y}{-\can{X}}}
\end{equation}
et
\begin{equation}\label{eq:zerr2j0j1:2}
\sum_{\becE\in \diveffc^{I\cup J}}
\abs{\mu_X(\becE)}
\!\!\!\!\!\!\!\!\!\!\!\!
\sum_{
\substack{
y\in\Pic(X)^{\vee}\cap \ceff(X)^{\vee}
\\
\acc{y}{\scG_j}\geq \deg(\ecG_j),\quad j\in J
\\
\acc{y}{\scF_i}\geq \deg(\ecF_i),\quad i\in I
\\
\acc{y}{\scG_{j_0}+\scG_{j_1}-\Dtot}
\leq 
\deg(\ecG_{j_0})+
\deg(\ecG_{j_1})+
2\,\gc-2
\\
\becD\in \diveffc^I\\
\deg(\ecD_i)=\acc{y}{\scF_i}-\deg(\ecF_i),\quad i\in I 
}
}
\!\!\!\!\!\!\!\!\!\!\!\!
q^{\card{J}-2}\,\wfonc(\becD,\becE,y)
\,
t^{\acc{y}{-\can{X}}}
\end{equation}

\begin{hyp}\label{hyp:zerr2:bis}
La s\'erie $Z((1-\frac{1}{\card{J}-1})\sum
  \scG_j,(1-\frac{1}{\card{J}-1})_{j\in J},t)$
(\cf l'\'enonc\'e du lemme \ref{lm:muxconv})
est $q^{-1}$-contr\^ol\'ee \`a l'ordre $\rg(\Pic(X))-1$.
\end{hyp}
\begin{rem}
D'apr\`es le lemme \ref{lm:muxconv} et  la formule d'adjonction \eqref{eq:adj},
l'hypoth\`ese \ref{hyp:zerr2:bis} est satisfaite
d\`es qu'il existe $(\bbf,\bg)\in (\bR_{\geq 0})^{I\cup J}$ tel que
\begin{equation}
((1+f_i)_{i\in I},\left(g_j+1-\frac{1}{\card{J}-1}
\right)_{j\in J})
\end{equation}
est $\mu_X$-convergent et  
\begin{equation}
\sumu{j}\left(\frac{1}{\card{J}-1}-g_j\right)\,\scG_j
-\Dtot-\sum_i f_i \scF_i\in \ceff(X)\setminus \{0\}.
\end{equation}
\end{rem}
\begin{rem}\label{rem:hyp:inter}
Supposons que l'intersection des $\{\scG_j\}_{j\in J}$ est non vide.
D'apr\`es la remarque \ref{rem:inter:bis}, l'hypoth\`ese
\ref{hyp:zerr2:bis} est alors
satisfaite d\`es qu'on a
\begin{equation}
\frac{1}{\card{J}-1}\sumu{j}\scG_j
-\Dtot\in \ceff(X)\setminus \{0\}.
\end{equation}
\end{rem}

\begin{prop}\label{prop:need:hyp:II}
Soit $j_0,j_1\in J$ tel qu'on ait 
\begin{equation}
\scG_{j_0}+\scG_{j_1}-\Dtot\in \ceff(X)\setminus \{0\}.
\end{equation}
On suppose que les hypoth\`ese \ref{hyp:conv:cEeps} et \ref{hyp:zerr2:bis} sont satisfaites.
Alors la s\'erie $\Zerr{2,j_0,j_1}$ est $q^{-1}$-contr\^ol\'ee \`a l'ordre $\rg(\Pic(X))-1$.
\end{prop}
\begin{proof}
En effet, sous l'hypoth\`ese \ref{hyp:conv:cEeps}, le r\'esultat est
v\'erifi\'e pour  la s\'erie \eqref{eq:zerr2j0j1:2} 
(qui coïncide \`a une constante pr\`es avec $\Zerr{1,j_0,j_1}$) et 
sous l'hypoth\`ese \ref{hyp:zerr2:bis} 
il l'est pour la s\'erie \eqref{eq:zerr2j0j1:2}.
\end{proof}

Afin d'appliquer les propositions \ref{prop:need:hyp:I} et
\ref{prop:need:hyp:II}, nous aurons besoin de l'hypoth\`ese suivante.
\begin{hyp}\label{hyp:pos}
Il existe une num\'erotation de $J$ telle qu'on ait 
\begin{equation}
\forall 1\leq j\leq \card{J}-1,\quad \scG_{j}+\scG_{j+1}-\Dtot\in \ceff(X)\setminus \{0\}.
\end{equation}
\end{hyp}
\subsection{Les termes $Z_K$, avec $K\neq J$}
\begin{prop}
Soit $K\subset J$ avec $K\neq J$.
On suppose que l'hypoth\`ese \ref{hyp:zerr2:bis} est satisfaite.
Alors la s\'erie $Z_{K}$ est $q^{-1}$-contr\^ol\'ee \`a l'ordre $\rg(\Pic(X))-1$.
\end{prop}
\begin{proof}
Si $K'\subset K$, il est imm\'ediat que la s\'erie $Z_{K'}$ est
major\'ee par $Z_{K}$. Il suffit donc de traiter le cas o\`u $K$ est \'egal
\`a $J\setminus \{j_0\}$.
La s\'erie consid\'er\'ee s'\'ecrit pour m\'emoire
\begin{equation}
\sum_{\becE\in \diveffc^{I\cup J}}
\abs{\mu_X(\becE)}
\sum_{
\substack{
y\in\Pic(X)^{\vee}\cap \ceff(X)^{\vee}
\\
\acc{y}{\scG_j}\geq \deg(\ecG_j),\quad j\in J
\\
\acc{y}{\scF_i}\geq \deg(\ecF_i),\quad i\in I
\\
\becD\in \diveffc^I\\
\deg(\becD)=(\acc{y}{\scF_i}-\deg(\ecF_i))}
}
\fonc_{J\setminus \{j_0\}}(\becD,\becE)
\,
t^{\acc{y}{-\can{X}}}
\end{equation}
D'apr\`es la proposition \ref{prop:maj:fonc:j},
cette s\'erie est major\'ee \`a une constante pr\`es par  
\begin{equation}
\sum_{\becE\in \diveffc^{I\cup J}}
\abs{\mu_X(\becE)}
\sum_{
\substack{
y\in\Pic(X)^{\vee}\cap \ceff(X)^{\vee}
\\
\acc{y}{\scG_j}\geq \deg(\ecG_j),\quad j\in J
\\
\acc{y}{\scF_i}\geq \deg(\ecF_i),\quad i\in I
\\
\becD\in \diveffc^I\\
\deg(\becD)=(\acc{y}{\scF_i}-\deg(\ecF_i))}
}
q^{\left(1-\frac{1}{\card{J}-1}\right)(\sumu{j\neq j_0}\acc{y}{\scG_j}-\deg(\ecG_j))}
t^{\acc{y}{-\can{X}}}
\end{equation}
et donc par la s\'erie \eqref{eq:zerr2j0j1:1}, pour laquelle 
le r\'esultat est v\'erifi\'e sous l'hypoth\`ese \ref{hyp:zerr2:bis}. 
\end{proof}

\section{Construction d'une famille de quadrique intrins\`eques et
  d\'emonstration
du r\'esultat principal}\label{sec:constr}
Le but de cette section est de construire une famille de vari\'et\'es
justiciables de l'application de la m\'ethode d\'ecrite dans la section pr\'ec\'edente.
\subsection{Une famille d'\'eventails projectifs et lisses}\label{subsec:famfan}
Soit $\indfan$ un ensemble fini de cardinal sup\'erieur \`a $3$ et $K$ un $\bZ$-module libre de 
base $(\{\scF_i\}_{i\in \indfan},\scF_0)$.
Pour $i\in \indfan$, on pose $\scG_i\eqdef-\scF_i+\scF_0+\sum_{j\in \indfan}\scF_j$.
Soit $\pi\,:\,\bZ^{(\indfan\times\{1,2\})\cup\{0\}}\to M$ le morphisme qui envoie
  $e_{i,1}$ sur $\scF_i$, $e_{i,2}$ sur $\scG_i$ et $e_{0}$ sur $\scF_0$.
On consid\`ere la suite exacte de $\bZ$-modules libres de rang fini
\begin{equation}
0\to M\overset{\iota}{\to} \bZ^{(\indfan\times\{1,2\})\cup \{0\}}\overset{\pi}{\to} K\to 0
\end{equation}
et la suite exacte duale
\begin{equation}
0\to K^{\vee}\to \bZ^{(\indfan\times\{1,2\})\cup \{0\}}\overset{\iota^{\vee}}{\to} M^{\vee}\to 0.
\end{equation}
Pour $i\in \indfan$ soit  $g_i\eqdef \iota^{\vee}(e_{i,2})$. Ainsi
$\{g_i\}_{i\in \indfan}$ est une base de $M^{\vee}$.
On note $\{x_i\}_{i\in \indfan}$ sa base duale. 
On note  \'egalement $h\eqdef \iota^{\vee}(e_0)=-\sum_{i\in \indfan}g_i$ et,
pour tout $i\in \indfan$, $f_i\eqdef \iota^{\vee}(e_{i,1})=h+g_i$.
\begin{prop}\label{prop:ex:un:sigma:n}
Aux automorphismes induits par les permutations des \'el\'ements de la
base $\{g_i\}$ pr\`es, il existe un unique \'eventail simplicial  de $M^{\vee}$
dont l'ensemble des g\'en\'erateurs des rayons est $\{h\}\cup\{f_i\}_{i\in \indfan}\cup\{g_i\}_{i\in \indfan}$. 
Cet \'eventail est projectif et lisse.
\end{prop}
Si $n$ est le cardinal de $\indfan$, cet \'eventail sera not\'e $\Sigma_n$.
On utilisera dans ce qui suit le lemme \'el\'ementaire suivant~:
\begin{lemme}\label{lm:calc:rg}
Soit $\indfan_1,\indfan_2\subset \indfan$. Alors le rang de la famille $\{(g_i)_{i\in
  \indfan_1}, (f_i)_{i\in \indfan_2}\}$ est $\card{\indfan_1\cup \indfan_2}$ si $\indfan_1\cap \indfan_2=\vide$
et $1+\card{\indfan_1\cup \indfan_2}$ sinon. En particulier cette famille est libre si et
seulement si on a $\card{\indfan_1\cap \indfan_2}=0$ ou $1$.
\end{lemme}
Pour toute partie $A$, on note $\cone{A}$ le c\^one engendr\'e par
cette partie.
Soit 
$
\ecC\eqdef\cone{\{g_i\}_{i\in \indfan}}=\cap_{i\in \indfan}\{x_i\geq 0\} 
$
et, pour $i\in \indfan$,
$\ecC_i\eqdef \cone{\{h\}\cup \{f_j\}_{j\neq i}}$.
On v\'erifie aussit\^ot que ces c\^ones sont simpliciaux et qu'on a, pour $i\in \indfan$,
\begin{equation}\label{eq:eq:cprj}
\ecC_i
=\{\sumu{j\neq i}x_j\leq (n-1)x_i\}\cap \bigcapu{j\neq i} \{x_j\geq x_i\}.
\end{equation}
Pour $\jndfan\subsetneqq \indfan$, on pose $\ecC_{\cJ}\eqdef \cone{\{g_i,f_i\}_{i\in \cJ}}$.

\begin{prop}\label{prop:descr:ev}
\begin{enumerate}
\item\label{item:i:prop:descr:ev}
Soit $i\in \indfan$. On a 
\begin{equation}
\ecC_{\indfan\setminus \{i\}}=\{x_i\leq 0\}\cap \{\sumu{j\neq i}x_j\geq
(n-2)x_i\}\cap \bigcapu{j\neq i} \{x_j\geq x_i\}
\end{equation}
\item\label{item:ii:prop:descr:ev}
Soit $\jndfan\subsetneqq \indfan$. Les faces 
de $\ecC_{\jndfan}$ sont les c\^ones $\ecC_{\kndfan}$, 
$\cone{\{g_i\}_{i\in \kndfan}}$ et $\cone{\{f_i\}_{i\in \kndfan}}$, pour $\kndfan$ d\'ecrivant
les parties de $\jndfan$.
\item\label{item:iii:prop:descr:ev}
Soit $\jndfan\subsetneqq \indfan$ et  $\{1,\dots,\card{\jndfan}\}\isom \jndfan$ 
une num\'erotation de $\jndfan$ (en d'autres termes un ordre total $\prec$ sur $\jndfan$). 
Alors la famille 
\begin{equation}
\scF_{\prec}\eqdef\{\cone{\{f_{i_1},\dots,f_{i_k},g_{i_k},g_{i_{k+1}},\dots,g_{i_{\card{\jndfan}}}\}}\}_{1\leq
  k\leq \card{\jndfan}}
\end{equation}
est l'ensemble des c\^ones maximaux d'une subdivision simpliciale
$\scS_{\prec}$ de $\ecC_{\jndfan}$.
L'application $\prec\mapsto \scS_{\prec}$ est une bijection de
l'ensemble des ordres totaux sur $\jndfan$ sur l'ensemble des subdivisions
simpliciales de $\ecC_{\jndfan}$.
\item\label{item:iv:prop:descr:ev}
Soit $\jndfan\subsetneqq \indfan$, $\kndfan\subset \jndfan$, 
$\prec$ un ordre total sur $\jndfan$ et
$\prec_{|_{\kndfan}}$ l'ordre induit sur $\cK$. Alors la
subdivision simpliciale induite par $\scS_{\prec}$ sur $\ecC_{\cK}$
est $\scS_{\prec_{|_{\kndfan}}}$.
\item\label{item:v:prop:descr:ev}
La famille $\{\ecC\}\cup \{\ecC_{\indfan\setminus \{i\}},\ecC_{i}\}_{i\in \indfan}$
est l'ensemble des c\^ones maximaux d'un \'eventail $\Sigma$ de $M^{\vee}$,
qui est l'\'eventail complet minimal 
dont un ensemble de g\'en\'erateurs des rayons
est $\{h\}\cup \{g_i,f_i\}_{i\in \indfan}$.
\item\label{item:vi:prop:descr:ev}
Soit $\prec$ un ordre total sur $\indfan$.
La famille de c\^ones 
$\{\ecC\}
\cup 
\{\ecC_{i}\}_{i\in \indfan}
\cup 
\cupu{i\in
\indfan} \scF_{\prec_{|_{\indfan\setminus \{i\}}}}$
est l'ensemble des c\^ones maximaux d'un \'eventail simplicial $\Sigma_{\prec}$  de $M^{\vee}$,
qui raffine $\Sigma$. L'application 
$\prec \mapsto \Sigma_{\prec}$ est une bijection de l'ensemble des ordres
totaux sur $\indfan$ sur l'ensemble des subdivisions simpliciales de $\Sigma$.
Pour tout ordre total $\prec$  sur $\indfan$, l'\'eventail $\Sigma_{\prec}$ est projectif et lisse.
\end{enumerate}
\end{prop}
\begin{proof}
On v\'erifie aussit\^ot l'inclusion 
\begin{equation}
\ecC_{\indfan\setminus \{i\}}\subset 
\{\sumu{j\neq i}x_j\leq (n-1)x_i\}\cap \bigcapu{j\neq i} \{x_j\geq x_i\}.
\end{equation}
R\'eciproquement, soit $x$ un \'el\'ement du c\^one de droite.
Il existe alors un $k\neq i$ tel qu'on ait
$x_k\geq \frac{n-2}{n-1}\,x_i$, et on a
\begin{equation}
x=
\left(x_k-\frac{n-2}{n-1}x_i\right)\,g_k
+\sum_{j\neq k,i}(x_j-x_i)\,g_j
-\frac{x_i}{n-1}\sum_{j\neq k,i}f_j\in \cone{\{g_j,f_j\}_{j\neq i}}.
\end{equation}
Le point \ref{item:i:prop:descr:ev} est donc d\'emontr\'e.

On constate aussit\^ot les \'egalit\'es 
\begin{equation}
\ecC_{\indfan\setminus \{i\}}\cap \{x_i=0\}=\cone{\{g_j\}_{j\neq i}}
\end{equation}
\begin{equation}
\ecC_{\indfan\setminus \{i\}}\cap \{\sumu{j\neq i}x_j=(n-1)x_i\}=\cone{\{f_j\}_{j\neq i}}
\end{equation}
\begin{equation}
\text{et}\quad\ecC_{\indfan\setminus \{i\}}\cap \{x_k=x_i\}=\cone{\{g_j,f_j\}_{j\notin \{i,k\}}}.
\end{equation}
Ces c\^ones \'etant de dimension $\card{\indfan}-1$ (les deux premiers sont simplicaux,
pour le dernier cela d\'ecoule du lemme \ref{lm:calc:rg}), 
on obtient ainsi, toujours d'apr\`es le lemme \ref{lm:calc:rg}, toutes les facettes de $\ecC_{\indfan\setminus \{i\}}$.
Le point \ref{item:ii:prop:descr:ev} en d\'ecoule en utilisant une
r\'ecurrence sur le cardinal de $\indfan$.

Soit $\jndfan\subsetneqq \indfan$.
Supposons d'abord que $\jndfan=\{k,\ell\}$ avec $k\neq \ell$. 
D'apr\`es le point \ref{item:ii:prop:descr:ev}
les facettes de
$\ecC_{\jndfan}$ sont alors $\cone{\{g_k,g_{\ell}\}}$,
$\cone{\{f_k,f_{\ell}\}}$,
$\cone{\{f_k,g_k\}}$
et $\cone{\{g_{\ell},f_{\ell}\}}$. Compte tenu du lemme
\ref{lm:calc:rg}, on en d\'eduit aussit\^ot 
l'assertion dans ce cas-l\`a.
Passons au cas g\'en\'eral. Soit $\scS$ 
une subdivision simpliciale de
$\ecC_{\jndfan}$. 
D'apr\`es le lemme \ref{lm:calc:rg}, tout c\^one maximal de $\scS$ s'\'ecrit 
$\cone{\{f_k,g_{\ell}\}_{k\in \jndfan_1,\ell\in \jndfan_2}}$ avec
$\jndfan_1\cup \jndfan_2=\jndfan$ 
et $\card{\jndfan_1\cap \jndfan_2}=1$.
Comme $\cone{\{g_i\}_{i\in \jndfan}}$ est une face de $\ecC_{\jndfan}$,
c'est \'egalement une face d'un c\^one maximal de $\scS$.
Un tel c\^one s'\'ecrit n\'ecessairement $\cone{\{g_i\}_{i\in \jndfan}\cup \{f_{i_0}\}}$
pour un certain $i_0\in \jndfan$. 
Supposons avoir
construit $r$ \'el\'ements $i_0,i_1,\dots, i_r$ de $\jndfan$   tels que $\scS$
contienne les c\^ones
$\cone{\{f_{i_0},f_{i_1},\dots,f_{i_k},g_{i_k}\}\cup\{g_i\}_{i\in \jndfan\setminus \{i_0,\dots,i_k\}}}$ 
pour $1\leq k\leq r$.
Si on a $r<\card{\indfan}$, le c\^one
\begin{equation}\label{eq:cone:inter}
\cone{\{f_{i_0},f_{i_1},\dots,f_{i_r},g_{i_r}\}\cup\{g_i\}_{i\in \jndfan\setminus \{i_0,\dots,i_k\}}}
\end{equation}
a pour face
$\cone{\{f_{i_0},\dots,f_{i_r}\}\cup\{g_i\}_{\jndfan\setminus\{i_0,\dots,i_r\}}}$, 
laquelle n'est incluse dans aucune des faces de $\ecC_{\jndfan}$. Il existe
donc un autre c\^one maximal de $\scS$ ayant
pour face le c\^one \eqref{eq:cone:inter}. Un tel c\^one maximal s'\'ecrit donc soit 
\begin{equation}
\cone{\{f_{i_0},\dots,f_{i_r},f_k\}\cup\{g_i\}_{i\in \jndfan\setminus\{i_0,\dots,i_r\}}}
\end{equation}
pour un $k\in \jndfan\setminus \{i_0,\dots,i_r\}$, soit
\begin{equation}
\cone{\{f_{i_0},\dots,f_{i_r}\}\cup
  \{g_i\}_{i\in \jndfan\setminus \{i_0,\dots,i_r\}}\cup \{g_k\}}
\end{equation}
pour un $k\in \jndfan\setminus \{i_0,\dots,i_r\}$. 
Mais cette derni\`ere possibilit\'e est exclue, car elle entra\^\i nerait que
les c\^ones $\cone{\{g_k,f_k,f_{i_r}\}}$ et $\cone{\{g_{i_r},f_k,f_{i_r}\}}$ 
sont deux c\^ones d'une m\^eme subdivision simpliciale de
$\cone{\{g_k,g_{i_r},f_k,f_{i_r}\}}$, ce qui contredirait le r\'esultat dans le cas $\card{\jndfan}=2$.
De proche en proche, on construit donc une num\'erotation
$i_0,\dots,i_{\card{\jndfan}}$ de $\jndfan$ telle 
que les c\^ones
\begin{equation}
\left\{
\cone{\{f_{i_0},f_{i_1},\dots,f_{i_k},g_{i_k},g_{i_{k+1}},\dots,g_{i_{\card{\jndfan}}}\}}
\right\}_{1\leq k\leq \card{\jndfan}}
\end{equation}
sont des c\^ones maximaux de $\scS$.
Supposons que $\scS$ contienne un c\^one maximal qui n'est pas dans
cette liste~: il existe alors $1\leq \ell<k\leq n$ tel que ce c\^one
 ait pour face $\cone{\{f_{i_k},g_{i_k},g_{i_{\ell}}\}}$.
Alors les c\^ones $\cone{\{f_{i_k},g_{i_k},g_{i_{\ell}}\}}$ et $\cone{\{f_{i_{\ell}},g_{i_{\ell}},g_{i_k}\}}$ 
sont deux c\^ones d'une m\^eme subdivision simpliciale de
$\cone{\{f_{i_k},f_{i_{\ell}},g_{i_k},g_{i_{\ell}}\}}$~: contradiction.
Compte tenu du fait qu'il existe au moins une subdivision simpliciale
de $\ecC_K$ et de la sym\'etrie du probl\`eme, le point
\ref{item:iii:prop:descr:ev} est d\'emontr\'e. Le point \ref{item:iv:prop:descr:ev}
est alors imm\'ediat.

D\'emontrons le point \ref{item:v:prop:descr:ev}.
Montrons d'abord que les c\^ones consid\'er\'es recouvrent $M^{\vee}\otimes {\bR}$.
Soit $x\in M^{\vee}\otimes {\bR}$. Soit $i$ tel que 
$x_i=\Min_{j\in \indfan} (x_j)$. 
Si on a $x_i\geq 0$, $x$ est dans $\ecC$. Sinon, d'apr\`es \eqref{eq:eq:cprj}
et le point \ref{item:i:prop:descr:ev},
$x$ est
dans $\ecC_{\indfan\setminus \{i\}}$ ou $\ecC_{i}$ selon le signe de
$\sum_{j\neq i}x_j-(n-1)x_i$. 

Par ailleurs, pour tout $i$ dans $\indfan$, l'hyperplan $\{x_i=0\}$
s\'epare $\ecC$ et $\ecC_{\indfan\setminus \{i\}}$ d'une part, $\ecC$ et $\ecC_{i}$ d'autre part.
Pour $i,k$ \'el\'ements distincts de $\indfan$, l'hyperplan $\{x_i=x_k\}$
s\'epare respectivement $\ecC_{\indfan \setminus \{i\}}$ et $\ecC_{\indfan\setminus \{k\}}$, $\ecC_{i}$ et
$\ecC_{k}$, $\ecC_{\indfan\setminus \{i\}}$ et $\ecC_{k}$. Enfin pour $i\in \indfan$, l'hyperplan
$\{\sumu{j\neq i}x_j=(n-1)x_i\}$ s\'epare $\ecC_{\indfan\setminus \{i\}}$ et $\ecC_{i}$.
Ceci ach\`eve la d\'emonstration du point \ref{item:v:prop:descr:ev}.
Le point \ref{item:vi:prop:descr:ev} d\'ecoule alors des points \ref{item:iii:prop:descr:ev}
et \ref{item:iv:prop:descr:ev}, exception faite de la lissit\'e de
l'\'eventail $\Sigma_{\prec}$ qui se v\'erifie aussit\^ot et de sa projectivit\'e.
Il suffit de
montrer cette derni\`ere propri\'et\'e pour un ordre total
$\indfan\isom\{1,\dots,n\}$, o\`u $n$ est le cardinal de $\indfan$.
On consid\`ere les c\^ones r\'eguliers de $K$ obtenus par dualit\'e de Gale
(\cf \cite[Section 3]{BerHau:bun}) \`a partir des
c\^ones maximaux de $\Sigma_{\prec}$~: 
\begin{equation}
\wt{\ecC}\eqdef\cone{\{\scF_i\}_{i\in \indfan},\scF_0}
\end{equation}
\begin{equation}
\wt{\ecC_i}\eqdef\cone{\{\scG_j\}_{j\in \indfan}\cup\scF_i\}}\quad i\in \indfan
\end{equation}
\begin{equation}
\text{et }\wt{\ecC_{i,j}}\eqdef\cone{\{\scF_i\}_{\substack{1\leq k\leq i-1\\k\neq j}}\cup
  \{\scG_k\}_{\substack{i+1\leq k\leq n\\k\neq j}}\cup\{\scF_j,\scG_j,\scF_0\}}
\quad i,j\in \indfan,\,i\neq j
\end{equation}
D'apr\`es \cite[Corollary 9.3]{BerHau:bun}, 
pour montrer que $\Sigma_{\prec}$ est projectif,
il suffit de montrer que les int\'erieurs relatifs de ces c\^ones 
ont une intersection non vide. 
On d\'etermine facilement les \'equations des int\'erieurs relatifs de ces c\^ones~:
\begin{equation}\label{eq:int:c}
\intrel(\wt{\ecC})=\capu{0\leq i \leq n} \{y_i>0\}
\end{equation}
\begin{equation}\label{eq:int:cpi:1}
\intrel(\wt{\ecC_i})
=
\{\sum_{1\leq j\leq n}y_j>(n-1)\,y_0
\}
\bigcap
\{
\sum_{\substack{
1\leq j\leq n
\\
j\neq i}}
y_j
>(n-2)\,y_0
\}
\bigcap
\capu{\substack{0\leq j \leq n,\\j\neq i}}
\{y_j<z\}
\end{equation}
\begin{multline} \label{eq:int:cpij:1}
\text{et }\intrel(\wt{\ecC_{i,j}})=
\capu{\substack{1\leq k\leq i-1,\\ k\neq j}}
\{y_i<y_k\}
\bigcap
 \capu{\substack{i+1\leq k\leq n,\\ k\neq j}}
\{y_i>y_k\}
\\
\bigcap
\{
(n-2-i)\,y_i<
\sum_{\substack{
i+1\leq k\leq n
\\
k\neq j}}
y_k
\}
\bigcap
\{
(n-1-i)\,y_i<
y_j+
\sum_{\substack{
i+1\leq k\leq n
\\
k\neq j}}
y_k
\}
\end{multline}
Soit $y_0>0$ et $\eps>0$. Pour $1\leq i\leq n$, posons $y_i\eqdef \frac{n-1}{n}\,y_0+(n-i)\,\eps$.
Alors pour $\eps>0$ assez petit on v\'erifie que $(y_0,y_1,\dots,y_n)$
est bien dans l'intersection de ces int\'erieurs relatifs.
\end{proof}
\subsection{Construction de vari\'et\'es d'anneau de Cox donn\'e}\label{subsec:BerHau}
Nous rappelons dans cette sous-section un r\'esultat de l'article
\cite{BerHau:Cox}, dont nous suivrons autant que faire se peut les notations.
Soit $k$ un corps, $\indsec$ un ensemble fini non vide
et $R$ un quotient de l'anneau de polyn\^omes $k[u_i]_{i\in \indsec}$. 
On suppose qu'on a $R^{\inv}=k^{\inv}$, que $R$ est factoriel, fid\`element
gradu\'ee par un $\bZ$-module libre de rang fini $K$, et  
que les images $\bar{u_i}$ des $u_i$ dans $R$ sont des \'el\'ements $K$-homog\`enes premiers non
deux \`a deux associ\'es. On note $\mfF\eqdef\{\bar{u_i}\}_{i\in \indsec}$.
L'application $(e_i)\mapsto \deg(\bar{u_i})$ induit
une suite exacte de $\bZ$-modules libres de rang fini
\begin{equation}
0\to M\overset{\iota}{\to}\bZ^{\,\indsec}\overset{\pi}{\to}K\to 0.
\end{equation}
On note $\gamma$ le c\^one $\bR_{\geq 0}^{\indsec}$ et, pour
$\jndsec\subset \indsec$, 
$\gamma_{\jndsec}$
la face de $\gamma$ engendr\'ee par les $\{e_i\}_{i\in \jndsec}$~; on identifiera
dans la suite l'ensemble des faces du c\^one $\gamma$ \`a l'ensemble des parties de $\indsec$.
Ainsi $(\bZ^{\indsec}\overset{\pi}{\to}K,\gamma)$ est le c\^one projet\'e
associ\'e au syst\`eme de g\'en\'erateurs $\mfF$ (au sens de \cite[p.1212]{BerHau:Cox}).
On suppose qu'il existe un  \'eventail $\Sigma$ de $M^{\vee}$ projectif et lisse
 dont les rayons sont les $\iota^{\vee}(e_i^{\vee})$.
Soit $\ecP_{\Sigma}$ l'ensemble des parties maximales $\jndsec$ de $\indsec$
telles que $\iota^{\vee}(\sum_{i\in \jndsec}\bR_{\geq 0}e_i^{\vee})$ est un c\^one maximal de
$\Sigma$. L'ensemble  $\{\indsec\setminus \jndsec\}_{\jndsec\in \ecP_{\Sigma}}$ forme alors la collection
couvrante $\Cov(\Theta_{\Sigma})$ d'une grappe $\Theta_{\Sigma}$ dans le c\^one projet\'e
$(\bZ^{\,\indsec}\overset{\pi}{\to}K,\bR_{\geq 0}^{\,\indsec})$, au sens de
\cite[p. 1209]{BerHau:Cox}. La grappe  $\Theta_{\Sigma}$ elle-m\^eme
est l'ensemble des c\^ones minimaux parmi les $\{\pi(\gamma_{\jndsec})\}_{\jndsec\in
  \Cov(\Theta_{\Sigma})}$. L'ensemble $\Rlv(\Theta_{\Sigma})$ est l'ensemble des parties de
$\indsec$ contenant un \'el\'ement de $\Cov(\Theta_{\Sigma})$. Ainsi $\Theta_{\Sigma}$ est aussi l'ensemble
des c\^ones minimaux parmi les $\{\pi(\gamma_{\jndsec})\}_{\jndsec\in   \Rlv(\Theta_{\Sigma})}$.

Pour $\jndsec\subset \indsec$, on pose 
$
\bA^{\,\indsec}_{\jndsec}\eqdef\{(u_i)\in \bA^{\,\indsec},\quad u_i\neq 0 \eq i\in \jndsec\}
$
et $\ecZ(\jndsec)\eqdef\bA^{\indsec}_{\jndsec}\cap \Spec(R)$. Une partie $\jndsec$ de
$\indsec$ est 
par d\'efinition une $\mfF$-face
 si $\ecZ(\jndsec)$ est non vide. 
L'ensemble des \'el\'ements de $\Rlv(\Theta_{\Sigma})$ qui sont des $\mfF$-face est l'ensemble $\Rlv(\Phi_{\mfF,\Sigma})$ des parties
pertinentes  d'une $\mfF$-grappe $\Phi_{\mfF,\Sigma}$ dans le c\^one projet\'e 
associ\'e \`a $\mfF$ \'etendue par $\Theta_{\Sigma}$ (au sens de \cite[Definition 2.3 et Definition 3.3]{BerHau:Cox}).
La $\mfF$-grappe $\Phi_{\mfF,\Sigma}$ elle-m\^eme est l'ensemble des c\^ones minimaux parmi les 
$\{\pi(\gamma_{\jndsec})\}_{\jndsec\in \Rlv(\Phi_{\mfF,\Sigma})}$.
L'ensemble $\Cov(\Phi_{\mfF,\Sigma})$ est l'ensemble des parties minimales de $\Rlv(\Phi_{\mfF,\Sigma})$.
Ainsi la $\mfF$-grappe $\Phi_{\mfF,\Sigma}$ est l'ensemble des c\^ones minimaux parmi les 
$\{\pi(\gamma_{\jndsec})\}_{\jndsec\in \Phi_{\mfF,\Sigma}}$. On notera qu'on
a  trivialement 
$\Cov(\Theta_{\Sigma})\subset \Rlv(\Phi_{\mfF,\Sigma})$. 
On pose alors 
\begin{equation}
\hat{X}_{\mfF,\Sigma}\eqdef\cupu{\jndsec\in
  \Rlv(\Phi_{\mfF,\Sigma})}\ecZ(\jndsec)
=
\cupu{\jndsec\in \Rlv(\Theta_{\Sigma})}\ecZ(\jndsec)\subset \bA^{\,\indsec}.
\end{equation}
\begin{thm}[Berchtold-Hausen]\label{thm:BerHau}
On conserve les hypoth\`eses et notations pr\'ec\'edentes.
On suppose que $\hat{X}_{\mfF,\Sigma}$ est lisse.
Il existe alors une vari\'et\'e $X_{\mfF,\Sigma}$ projective et lisse,
un isomorphisme $\Pic(X_{\mfF,\Sigma})\isom K$ et un isomorphisme
$\Cox(X_{\mfF,\Sigma})\isom R$
compatibles aux graduations de $\Cox(X_{\mfF,\Sigma})$ et $R$ par
$\Pic(X_{\mfF,\Sigma})$ et $K$, respectivement.
La vari\'et\'e $X_{\mfF,\Sigma}$ s'obtient comme le quotient g\'eom\'etrique
de $\hat{X}_{\mfF,\Sigma}$ sous l'action naturelle du tore $\Hom(K,\bG_m)$.
\end{thm}
Ceci d\'ecoule de \cite[Proposition 3.2, Theorem 4.2 et Proposition 5.6]{BerHau:Cox}.

\begin{rem}\label{rem:incidence}
On conserve les notations pr\'ec\'edentes.
Pour $i\in \indsec$, soit $\scE_i$ le diviseur de $X_{\mfF,\Sigma}$
donn\'e par l'annulation de la section $u_i$. Les relations d'incidence
des diviseurs $\scE_i$ se retrouvent ais\'ement \`a partir des donn\'ees ci-dessus.
Plus pr\'ecis\'ement, pour toute partie $\jndsec$ de $\indsec$, on a
$\cap_{i\in \jndsec}\,\scE_i\neq \vide$ si et seulement si 
il existe un \'el\'ement de $\Rlv(\Phi_{\mfF,\Sigma})$ disjoint de
$\jndsec$ 
si et seulement si 
il existe $\jndsec\subset\kndsec$ tel 
$\iota^{\vee}(\sum_{i\in \kndsec}\bR_{\geq 0}e_i^{\vee})$ est un c\^one
maximal de $\Sigma$ et il existe $(u_i)\in \Spec(R)$ tel que $u_i=0$
pour $i\in \jndsec$ et $\prod_{i\notin \kndsec}u_i\neq 0$
\end{rem}

\subsection{Le r\'esultat principal}\label{subsec:princ}
On conserve les notations  de la sous-section 
\ref{subsec:BerHau}.
Soit $n\geq 3$ un entier et
soit $R_n\eqdef k[\{s_i,t_i\}_{1\leq i\leq n},s_0]/\sum_{1\leq i\leq n} s_i\,t_i$. 
Soit $\mfF_n$ l'ensemble des images des \'el\'ements
$s_0,\{s_i,t_i\}_{1\leq i\leq n}$ dans $R_n$.
Soit $K_n$ le $\bZ$-module libre de base $(\{\scF_i\}_{0\leq i\leq n})$
On d\'efinit une $K$-graduation fid\`ele sur $R$ en posant $\deg(s_0)=\scF_0$ et,
pour $1\leq i\leq n$, $\deg(s_i)=\scF_i$ et
$\deg(t_i)=\scG_i\eqdef-\scF_i+\scF_0+\sum_{1\leq j \leq n}\scF_j$.
Soit $X_n$ la vari\'et\'e $X_{\mfF_n,\Sigma_n}$ o\`u $\Sigma_n$ est l'\'eventail de la proposition
On v\'erifie ais\'ement \`a partir de la description de $\Sigma_n$ 
et de l'\'equation d\'efinissant $R_n$ que la vari\'et\'e $\hat{X}_{\mfF_n,\Sigma_n}$ est lisse.
D'apr\`es le th\'eor\`eme \ref{thm:BerHau}, $X_n$ est une vari\'et\'e projective et lisse d'anneau de
Cox isomorphe \`a $R_n$ et de groupe de Picard isomorphe \`a $K_n$, tel
que le diviseur de $s_i$ est $\scF_i$ et le diviseur de $t_i$ est $\scG_i$.
On notera que $X_n$ est de dimension $n-1$. 
\begin{rem}\label{rem:inc:xn}
On v\'erifie ais\'ement, \`a partir de la remarque \ref{rem:incidence} que
l'intersection $\cap_{1\leq i\leq n}\scG_i$ est non vide
et que l'intersection $\cap_{1\leq i\leq n}\scF_i\cap \scG_i$ est vide.
\end{rem}
\begin{rem}\label{rem:comp}
D'apr\`es \cite{Has:eq:ut:cox:rings}, $X_3$ est isomorphe au plan projectif
\'eclat\'e en trois points align\'es, et est donc une compactification
\'equivariante de l'espace affine $\bG_a^2$. Plus g\'en\'eralement, $X_n$
est une compactification \'equivariante de $\bG_a^{n-1}$~: 
munissons en effet $\bP^{n-1}$ de coordonn\'ees homog\`enes
$(x_0:\dots:x_{n-1})$ et de l'action alg\'ebrique de $\bG_a^{n-1}$ qui est l'action
par translation sur $\{x_0\neq 0\}\isom \bG_a^{n-1}$ et est triviale sur $\{x_0=0\}$~;
on v\'erifie ais\'ement que le morphisme 
\beq
(s_0,\dots,s_{n},t_1,\dots,t_n)\mapsto (\prod_{0\leq i\leq n}s_i,s_1\,t_1,\dots,s_n\,t_n)
\eeq
induit par passage au quotient un morphisme $\pi_n\,:\,X_n\to \bP^{n-1}$ (qui
n'est autre que le morphisme d\'efini par la base $\{\prod_{0\leq i\leq n}s_i,s_1\,t_1,\dots,s_{n-1}\,t_{n-1}\}$ 
de $\struct{X_n}(\scF_0+\dots+\scF_n)$) et que $\pi_n$ induit un
isomorphisme de $X_n\setminus \cup_{0\leq i\leq n}\,\scF_i$ sur $\{x_0\neq
0\}$~; par ailleurs on v\'erifie que le morphisme
\begin{multline}
(\lambda_1,\dots,\lambda_{n-1})\,(s_0,\dots,s_{n},t_1,\dots,t_n)
\\
\mapsto 
(s_0,\dots,s_n,t_1+\lambda_1\prod_{i\neq
  1}s_i,\dots,t_{n-1}+\lambda_{n-1}\prod_{i\neq n-1}s_i,t_n-(\lambda_1+\dots+\lambda_{n-1})\,\prod_{i\neq n}s_i)
\end{multline}
induit une action de $\bG_a^{n-1}$ sur $X_n$, pour laquelle le morphisme $\pi_n$
est \'equivariant.
\end{rem}

\begin{thm}\label{thm:princ}
On conserve les notations pr\'ec\'edentes. 
Soit $n\geq 3$. 
Soit $Z_{X_n,-\can{X_n},X_{n,0}}(t)$ la fonction z\^eta des hauteurs
anticanonique associ\'ee \`a l'ouvert 
$X_{n,0}\eqdef X_n\setminus 
\cup_{0\leq i\leq n}\,\scF_i
\cup_{1\leq i\leq n}\,\scG_i$.
Alors la s\'erie
\begin{equation}
Z_{X_n,-\can{X_n},X_{n,0}}(t)-
\gamma(X_n)\!\!
\sum_{y\in \ceff(X_n)^{\vee}\cap \Pic(X_n)^{\vee}}\,(q\,t)^{\,\acc{y}{-\can{X_n}}}
\end{equation}
est $q^{-1}$-contr\^ol\'ee \`a l'ordre $\rg(\Pic(X_n))-1$.
\end{thm}
\begin{proof}
On veut appliquer le th\'eor\`eme \ref{thm:synth}. 
Remarquons que $\ceff(X_n)$ est simplicial, engendr\'e par la base
$\{\scF_i\}_{0\leq i\leq n}$.
Pour $1\leq i\leq n-1$, on a 
\begin{equation}
\scG_{i}+\scG_{i+1}-\Dtot=\sum_{\substack{0\leq j\leq n\\ j\notin\{i,i+1\}}}\scF_j\in \ceff(X_n)\setminus \{0\}
\end{equation}
et donc l'hypoth\`ese \ref{hyp:pos} vaut.
Par ailleurs on a
\begin{equation}
\frac{1}{n-1}\sum_{1\leq i\leq
  n}\scG_i-\Dtot=\frac{n}{n-1}\scF_0+\sum_{1\leq i\leq
  n}\scF_i-\sum_{0\leq i\leq n}\scF_i=\frac{1}{n-1}\scF_0\in
\ceff(X_n)\setminus \{0\}.
\end{equation}
D'apr\`es la remarque \ref{rem:hyp:inter} et le fait que l'intersection
des diviseurs $\{\scG_i\}_{1\leq i\leq n}$ est non vide, l'hypoth\`ese \ref{hyp:zerr2:bis}
est donc satisfaite.

Montrons \`a pr\'esent que l'hypoth\`ese \ref{hyp:conv:cEeps} est v\'erifi\'ee.
En reprenant les notations de la sous-section
\ref{subsubsec:estim:fonc1}, on a ici
\begin{equation}
F_{\rho,\be}(\bt)
=
\sum_{\bd\in \bN^{n+1}}
\rho^{\Infu{1\leq i\leq n}(d_i+f_i+g_i)}
\bt^{\,\bd}.
\end{equation}
D'apr\`es \cite[Proposition 3.2]{Bou:compt}, on a
\begin{equation}
\wt{F}_{\rho,0}=\frac{1-\produ{1\leq i\leq
    n}t_i}{1-\rho\,\produ{1\leq i\leq n}t_i}
\quad
\text{et}
\quad
\wt{F}_{\rho,\be}=\frac{G_{\rho,\be}(\bt)}{1-\rho\,\produ{1\leq i\leq n}\,t_i}
\end{equation}
o\`u $G_{\rho,\be}(\bt)=\sum Q_{\be,\bd}(\rho)\,\bt^{\bd}$ est un polyn\^ome en
$\bt$ avec la propri\'et\'e que le degr\'e de $Q_{\be,\bd}(\rho)$ est major\'e
par $\Min(d_i+f_i+g_i)$. 
On en d\'eduit en particulier que l'hypoth\`ese \ref{hyp:conv:sergen}
est satisfaite.

D'apr\`es le lemme \ref{lm:majem} ci-dessous, on
a pour tout $\be=(\bbf,\bg)\in \{0,1\}^{2\,n+1}$ et tout $\bd$ la majoration
\beq
\Min(d_i+f_i+g_i)\leq \abs{\bd}+\Min(f_i+g_i).
\eeq

D'apr\`es la remarque \ref{rem:hyp:conv:cEeps},
il suffit pour v\'erifier l'hypoth\`ese \ref{hyp:conv:cEeps} de montrer
que si $\be\in \{0,1\}^{2\,n+1}$ est non nul et tel que 
$\mu_{X_n}^0(\be)$ est non nul alors
 on a 
\begin{equation}\label{eq:maj:a:verif}
\Minu{1\leq i\leq
    n}(f_i+g_i)-\sum_{1\leq i\leq n} (f_i+g_i)-f_0\leq -2.
\end{equation}
Mais comme on a $n\geq 3$, \eqref{eq:maj:a:verif} est v\'erifi\'ee 
sauf dans la situation suivante~: tous les  $f_i+g_i$ sont nuls (et
dans ce cas $f_0=1$) sauf peut-\^etre un qui vaut $1$
(et dans ce cas $f_0=0$).
D'apr\`es la proposition \ref{prop:mu}, cette situation entra\^\i ne que $\mu^0_{X_n}(\bbf,\bg)$ est nul.
Ainsi l'hypoth\`ese \ref{hyp:conv:cEeps} est v\'erifi\'ee.

Montrons finalement que l'hypoth\`ese \ref{hyp:rel:term:princ} est v\'erifi\'ee.
Pour $(\bbf,\bg)\in \{0,1\}^{2\,n+1}$ posons, comme dans
\cite{Bou:compt},
\begin{multline}
\fact_{X_n,v}(\bbf,\bg)
\eqdef
q_v^{-\sumu{0\leq i\leq n} f_i-\sumu{1\leq i\leq n} g_i}
\wt{F}_{q_v,\bbf,\bg}(q_v^{-1})
\\
=(1-q_v^{-1})^n\,q_v^{\,-f_0}\,
\sum_{\substack{\bd\in \bN^n\\ d_i\geq f_i+g_i}}
q_v^{\Min(d_i)}\,q_v^{-\sumu{1\leq i\leq n}d_i},
\end{multline}
\begin{equation}
\kappa_v^{\bbf,\bg}\eqdef
\{(x_i)_{0\leq i\leq n},(y_i)_{1\leq i\leq n}\in \kappa_v^{2\,n+1},\quad
\forall i,\quad x_{i}=0\,\text{ si }\,f_{i}=1\text{ et }
y_{i}=0\,\text{ si }\,g_{i}=1
\},
\end{equation}
\begin{equation}
\text{et}\quad
\dens_{X_n,v}(\bbf,\bg)
\eqdef
\frac{
\card{
\{((x_i)(y_i))\in \kappa_v^{\bbf,\bg},\quad \sum_{1\leq i\leq n}x_i\,y_i=0\}
}}
{
q_v^{\,2\,n}
}.
\end{equation}
D'apr\`es \cite[Lemme 1.25]{Bou:compt}, pour montrer que l'hypoth\`ese \ref{hyp:rel:term:princ} est
v\'erifi\'ee, il suffit de montrer la relation
\begin{equation}\label{eq:rel:a:montrer}
\sum_{(\bbf,\bg)\in \{0,1\}^{2\,n+1}}
\mu_{X_n}^0(\bbf,\bg)\,\fact_{X_n,v}(\bbf,\bg)
 =
\sum_{(\bbf,\bg)\in \{0,1\}^{2\,n+1}}
\mu_{X_n}^0(\bbf,\bg)\,\dens_{X_n,v}(\bbf,\bg).
\end{equation}
Pour $f_0\in \{0,1\}$, on a d'apr\`es la remarque \ref{rem:inc:xn}
et la proposition \ref{prop:mu} la relation
\begin{equation}
\sum_{(\bbf,\bg)\in \{0,1\}^{2\,n}}\mu^0_{X_n}(f_0,\bbf,\bg)=0.
\end{equation}
Pour montrer la relation \eqref{eq:rel:a:montrer},
il suffit donc
de montrer pour tout $(\bbf,\bg)\in \{0,1\}^{2\,n}$ et tout $f_0\in
\{0,1\}$ la relation
\begin{multline}
\fact_{X_n,v}(f_0,f_1,\dots,f_n,\bg)-\fact_{X_n,v}(f_0,1,\dots,1,1,\dots,1)
\\
=
\dens_{X_n,v}(f_0,f_1,\dots,f_n,\bg)-\dens_{X_n,v}(f_0,1,\dots,1,1,\dots,1)
\end{multline}
et pour montrer cette derni\`ere relation il suffit par sym\'etrie de
montrer qu'on a 
\begin{multline}\label{eq:a:montrer}
\fact_{X_n,v}(f_0,0,f_2,\dots,f_n,\bg)-\fact_{X_n,v}(f_0,1,f_2\dots,f_n,\bg)
\\
=
\dens_{X_n,v}(f_0,0,f_2,\dots,f_n,\bg)-\dens_{X_n,v}(f_0,1,f_2\dots,f_n,\bg)
\end{multline}
Supposons $g_1=0$. Alors le membre de gauche de \eqref{eq:a:montrer}
vaut
\begin{equation}
(1-q_v^{-1})^n\,q_v^{\,-f_0}\,
\sum_{\substack{\bd\in \bN^{n-1}\\ d_i\geq f_i+g_i,\,\,2\leq i\leq n}}
\,q_v^{-\sumu{2\leq i\leq n}d_i}=(1-q_v^{-1})\,q_v^{\,-f_0-\sumu{2\leq
    i\leq n} f_i+g_i}
\end{equation}
tandis que le membre de droite vaut
\begin{align}
q_v^{\,(1-f_0)-2\,n}
\card{
\{\bx,\by\in \kappa_v^{(0,f_2,\dots,f_n),(0,g_2,\dots,g_n)},\,x_1\neq
0, \sum x_i\,y_i=0
\}
}
\hskip-0.4\textwidth
&
\\
&=q_v^{\,1-f_0-2\,n}\,(q_v-1)\,\card{\kappa_v^{(f_2,\dots,f_n),(g_2,\dots,g_n)}}
\\
&=q_v^{\,-f_0-2\,(n-1)}\,(q_v^{-1}-1)\,q_v^{\sumu{2\leq i \leq n}(1-f_i+1-g_i)}
\end{align}
d'o\`u l'\'egalit\'e dans ce cas-l\`a.

Supposons \`a pr\'esent qu'on a $g_1=1$. On peut \'egalement supposer qu'on
a, pour un certain entier $2\leq k\leq n$, $f_i+g_i\geq 1$ pour $2\leq
i\leq k$ et $f_i+g_i=0$ pour $i\geq k+1$.
Commen\c cons par la remarque suivante~:
posons pour $n\geq 1$ 
\begin{equation}
N_n(q_v)\eqdef\card{\{(x_1,\dots,x_n,y_1,\dots,y_n)\in \kappa_v^{2\,n},\quad
\sum_{1\leq i \leq n} x_i\,y_i=0\}}
\end{equation}
On v\'erifie aussit\^ot la relation 
\begin{equation}
N_n(q_v)=q_v\,N_{n-1}(q_v)+(q_v-1)\,q_v^{\,2\,n-2}
\end{equation}
qui entra\^\i ne par une r\'ecurrence imm\'ediate l'\'egalit\'e
\begin{equation}
\sum_{\bd\in \bN^n}q_v^{\,\Min(1,n_i)}q_v^{-\sum n_i}=\frac{q_v^{-2\,n+1}}{(1-q_v^{-1})^n}\,N_n(q_v).
\end{equation}
Le membre de gauche de \eqref{eq:a:montrer} vaut alors 
\begin{align}
(1-q_v^{-1})^n\,q_v^{\,-f_0-1}\,
\sum_{\substack{(d_2,\dots,d_n)\in \bN^{n-1}\\ d_i\geq f_i+g_i}}
q_v^{\Min(1,d_2,\dots,d_n)}\,q_v^{-\sumu{1\leq i\leq n}d_i}
\hskip-0.6\textwidth&
\\
&
=
(1-q_v^{-1})^{n-k+1}\,q_v^{\,-f_0-1-\sumu{2\leq i\leq k}(f_i+g_i)}
\sum_{\substack{(d_{k+1},\dots,d_n)\in \bN^{n-k}}}
q_v^{\,\Min(1,d_{k+1},\dots,d_n)}\,q_v^{-\sumu{k+1\leq i\leq n}d_i}
\\
&
=
(1-q_v^{-1})\,q_v^{\,-2\,(n-k)-f_0-\sumu{2\leq i\leq k}(f_i+g_i)}\,N_{n-k}(q_v)
\end{align}
et le membre de droite
\begin{multline}
q_v^{\,(1-f_0)-2\,n}
\card{
\{(\bx,\by)\in \kappa_v^{(0,f_2,\dots,f_n),(1,g_2,\dots,g_n)},
\,
x_1\neq 0,\, 
\sum_{k+1\leq i\leq n} x_i\,y_i=0
\}}
\\
=
q_v^{1-f_0-2\,n}(q_v-1)\,q_v^{\,\sumu{2\leq i\leq k}(1-f_i+1-g_i)}\,N_{n-k}(q_v)
\end{multline}
d'o\`u l'\'egalit\'e dans ce cas \'egalement.

Ceci ach\`eve la d\'emonstration du fait que  l'hypoth\`ese
\ref{hyp:rel:term:princ}
est v\'erifi\'ee pour la vari\'et\'e $X_n$, et donc la d\'emonstration du
th\'eor\`eme
\ref{thm:princ}.
\end{proof}

\begin{lemme}\label{lm:majem} 
Soit $n\geq 2$ un entier et $\bnu\in \bN^n$. 
On a pour tout $\bd\in \bN^n$ la majoration
\beq
\Min(d_i+\nu_i)\leq \abs{\bd}+\Min(\nu_i).
\eeq
\end{lemme}
\begin{proof}
On se ram\`ene aussit\^ot au cas o\`u $\Min(\nu_i)=0$, $n=2$, et $d_1+\nu_1=\Min(d_i+\nu_i)$.
Si $\nu_1=0$ la majoration est \'evidemment v\'erif\'ee.
Sinon on a $\nu_2=0$ et $d_2\geq d_1+\nu_1$ et l\`a encore la majoration
est v\'erifi\'ee.
\end{proof}

\section{Application \`a une surface de del Pezzo g\'en\'eralis\'ee}\label{sec:app:sdp}
\subsection{Une variante de la m\'ethode pr\'ec\'edente}

On consid\`ere \`a pr\'esent une l\'eg\`ere variante de la situation consid\'er\'ee
dans la section \ref{sec:hyp:int}, en supposant cette fois qu'on a un isomorphisme
\begin{equation}
\Cox(X)\longisom k[(s_i)_{i\in I},(t_j)_{j\in J}]/F(s_i,t_j)
\end{equation}
o\`u $F$ est $\Pic(X)$-homog\`ene de degr\'e $\Dtot$ et  s'\'ecrit
\begin{equation}
F=t_{j_0}^2\,\prod_{i\in I}s_i^{b_{i,j_0}}+\sum_{j\in J\setminus \{j_0\}} t_j\,\prod_{i\in I}s_i^{b_{i,j}},\quad b_{i,j}\in \bN.
\end{equation}
Par souci de simplification, on supposera en outre qu'on a
$J=\{1,2,3\}$ (en d'autres termes que $X$ est une surface).
On choisira $j_0=1$. 
On reprend alors
la strat\'egie d\'evelopp\'ee dans la section \ref{sec:hyp:int} pour aboutir
au r\'esultat suivant, que nous appliquerons ensuite \`a une surface de del Pezzo
g\'en\'eralis\'ee issue de la liste \'etablie par Derenthal dans \cite{Der:sdp:ut:hyp}.
\begin{thm}\label{thm:synth:bis}
On conserve les notations et hypoth\`eses pr\'ec\'edentes.
On suppose que les hypoth\`eses 
\ref{hyp:synth:dP6A2} ci-dessous sont satisfaites.
Alors la s\'erie
\begin{equation}
Z_{X,-\can{X},X_0}(t)-
\gamma(X)\!\!
\sum_{y\in \ceff(X)^{\vee}\cap \Pic(X)^{\vee}}\,(q\,t)^{\,\acc{y}{-\can{X}}}
\end{equation}
est $q^{-1}$-contr\^ol\'ee \`a l'ordre $\rg(\Pic(X))-1$.
\end{thm}
\begin{cor}\label{cor:appl}
On suppose que $X$ est la d\'esingularisation minimale de la surface de
del Pezzo singuli\`ere de degr\'e $6$ (respectivement $5$) avec une
singularit\'e de type $A_2$ (respectivement de type $A_3$). Alors $X$
v\'erifie la conclusion du th\'eor\`eme ci-dessus.
\end{cor}
La d\'emonstration du th\'eor\`eme \ref{thm:synth:bis} est tr\`es similaire \`a celle 
du th\'eor\`eme \ref{thm:synth} d\'evelopp\'ee dans la section \ref{sec:hyp:int}. 
Nous nous contentons d'indiquer succinctement les modifications n\'ecessaires.
Le lemme de comptage de sections globales (proposition
\ref{prop:compt}) 
est remplac\'e par la proposition \ref{prop:compt:bis} ci-dessous.
On d\'eduit de cette nouvelle version la proposition
\ref{prop:maj:fonc:dP6A2}, qui remplace la proposition \ref{prop:maj:fonc:j}.
Compte tenu de cette nouvelle proposition, la s\'erie g\'en\'eratrice
$F_{\rho,\be}$ 
d\'efinie par \eqref{eq:def:Frhoalphabeta} s'\'ecrit \`a pr\'esent
\begin{multline}\label{eq:def:Frhoalphabeta:dP6A2}
F_{\rho,\be}(\bt)
\\
\eqdef
\sum_{\bd\in \bN^I}
\rho^{
\Infu{2\leq j\leq 3}
\left(
g_j+
\sum_i
  b_{i,j}(d_i+f_i)
\right)
-\lceil
\frac{1}{2}\,\,\Infu{2\leq j\leq 3}
\left(
g_j+
\sum_i
  b_{i,j}(d_i+f_i)
\right)
-
\frac{1}{2}\,\,\Infu{1\leq j\leq 3}
\left(
\eps_j\,g_j+
\sum_i b_{i,j}(d_i+f_i)
\right)
\rceil
}
\bt^{\,\bd},
\end{multline}
o\`u $\eps_1=2$, $\eps_2=\eps_3=1$, et, pour $x\in \bR$, $\lceil x
\rceil$ d\'esigne le plus petit entier
sup\'erieur \`a $x$.
On aboutit alors naturellement \`a de nouvelles hypoth\`eses suffisantes 
(hypoth\`eses \ref{hyp:synth:dP6A2})
pour mener \`a bien la m\'ethode dans ce nouveau cadre.

Pour $\ecD\in \diveffc\otimes \bR$ on note
$\lceil \ecD \rceil\eqdef \sum_v \lceil v(\ecD) \rceil\,v$.
Par une d\'emarche semblable \`a celle utilis\'ee dans la d\'emonstration de
la proposition \ref{prop:compt}, on aboutit au r\'esultat suivant.
\begin{prop}\label{prop:compt:bis}
Soient $\ecH,\ecH_1,\ecH_2,\ecH_3,\ecH'_1,\ecH'_2$ et $\ecH'_3$ des diviseurs de $\courbe$
tels que les diviseurs $\ecH$, $\ecH_1+2\,\ecH'_1$ et $\ecH_j+\ecH'_j$
pour $j=2,3$ sont deux \`a deux lin\'eairement \'equivalents.
Pour $j\in \{1,2,3\}$, soit $s_j$  une section globale non nulle de $\OdeC(\ecH_j)$.
On fixe des isomorphismes 
\begin{equation}
\OdeC(\ecH)\isom \OdeC(\ecH_1+2\,\ecH'_1),\quad
\OdeC(\ecH)\isom \OdeC(\ecH_j+\ecH'_j),
\quad j\in \{2,3\}
\end{equation} 
ce qui permet pour $j\in \{2,3\}$ de d\'efinir l'application
\begin{equation}
\varphi_{s_1,\{s_k\}_{2\leq k\leq j}}\,:\,H^0(\courbe,\OdeC(\ecH'_1))
\times \prod_{2\leq k\leq j}H^0(\courbe,\OdeC(\ecH'_k))
\longto H^0(\courbe,\OdeC(\ecH))
\end{equation}
qui \`a $(t_1,\{t_k\}_{2\leq k\leq j})$ associe $t_1^2\,s_1+\sum_{2\leq k\leq j}t_k\,s_k$.
On note 
\begin{equation}
\Delta_{s_1,\{s_k\}_{2\leq k\leq j}}\eqdef \log_q \card{\varphi_{s_1,\{s_k\}_{2\leq k\leq j}}^{-1}(\{0\})}.
\end{equation}
\begin{enumerate}
\item
On a la majoration
\beq
\Delta_{s_1,s_2}\leq 1+\frac{1}{2}\deg(\ecH'_1)+\frac{1}{2}\,\deg(\ecH'_2).
\eeq
\item 
La quantit\'e $\Delta_{s_1,s_2,s_3}$ est major\'ee soit par 
\beq
1+\deg(\ecH'_1),
\eeq
soit par 
\begin{multline}
2+
\sum_{1\leq j\leq 3}\deg(\ecH'_j)-\deg(\ecH)+\deg(\Inf(\ddiv(s_2),\ddiv(s_3)))
\\
-
\deg\left(\lceil \frac{1}{2}\Inf(\ddiv(s_2),\ddiv(s_3))-\frac{1}{2}\Inf(\ddiv(s_1),\ddiv(s_2),\ddiv(s_3))\rceil\right).
\end{multline}
\item
On suppose qu'on a
\begin{equation}
\deg(\ecH'_2)+\deg(\ecH'_3)\geq \deg(\ecH)-\deg(\Inf(\ddiv(s_2),\ddiv(s_3)))+2\,\gc-1
\end{equation}
et 
\begin{equation}
\deg(\ecH'_1)\geq 
\deg\left(\lceil \frac{1}{2}\Inf(\ddiv(s_2),\ddiv(s_3))-\frac{1}{2}\Inf(\ddiv(s_1),\ddiv(s_2),\ddiv(s_3))\rceil\right)+2\,\gc-1.
\end{equation}
Notons que ceci est en particulier v\'erifi\'e si l'on a 
\begin{equation}\label{hyp:cor:clef:part:1:dP6A2}
\deg(\ecH'_2)+\deg(\ecH'_3)-\deg(\ecH)\geq 2\,\gc-1
\end{equation}
\begin{equation}\label{hyp:cor:clef:part:2:dP6A2}
\text{et}\quad\deg(\ecH'_1)+\frac{1}{2}\,\deg(\ecH'_2)+\frac{1}{2}\,\deg(\ecH'_3)-\deg(\ecH)\geq 2\,\gc-1.
\end{equation}
Alors $\Delta_{(s_1,s_2,s_3)}$
vaut 
\begin{multline}
2\,(1-\gc)+
\sum_{1\leq j\leq 3}\deg(\ecH'_j)-\deg(\ecH)+\deg(\Inf(\ddiv(s_2),\ddiv(s_3)))
\\
-
\deg\left(\lceil \frac{1}{2}\Inf(\ddiv(s_2),\ddiv(s_3))-\frac{1}{2}\Inf(\ddiv(s_1),\ddiv(s_2),\ddiv(s_3))\rceil\right).
\end{multline}
\end{enumerate}
\end{prop}
\begin{prop}\label{prop:maj:fonc:dP6A2}
Soit $\becE\in \diveffc^{I\cup J}$, 
 $\becD\in \diveffc^I$ et 
$y\in\Pic(X)^{\vee}\cap \ceff(X)^{\vee}$
tels que pour $i\in I$ on ait $\deg(\ecD_i)=\acc{y}{\scF_i}-\deg(\ecF_i)$.
\begin{enumerate}
\item\label{item:prop:maj:fonc:ijk:dP6A2}
Pour tout $\{i,j,k\}=\{1,2,3\}$, on  a
\begin{equation}\label{eq:maj:fonc:ijk:dP6A2}
\log_q \fonc_{\{1,2,3\}\setminus \{i\}}(\becD,\becE)\leq 1+\frac{1}{2}\,(\acc{y}{\scG_j}-\deg(\ecG_j))+\frac{1}{2}\,(\acc{y}{\scG_k}-\deg(\ecG_k)).
\end{equation}
\item\label{item:prop:maj:fonc:princ:dP6A2}
Posons $\eps_1=2$ et $\eps_2=\eps_3=1$.
La quantit\'e $\log_q \fonc_{\{1,2,3\}}(\becD,\becE)$ est major\'ee soit par
\begin{equation}\label{eq:maj:dP6A2:N123:bis}
1+\acc{y}{\scG_1}-\deg(\ecG_1),
\end{equation}
soit par
\begin{multline}
2+\acc{y}{\sum_{1\leq j\leq 3}\scG_j-\Dtot}-\sum_{1\leq j\leq 3} \deg(\ecG_j)
+
\deg(\Infu{2\leq j\leq 3}(
\sumu{i\in I} b_{i,j} (\ecF_i+\ecD_i)+\ecG_j
))
\\
-
\deg\left(\lceil \frac{1}{2}
\Infu{2\leq j\leq 3}(
\sumu{i\in I} b_{i,j} (\ecF_i+\ecD_i)+\ecG_j
)-\frac{1}{2}
\Infu{1\leq j\leq 3}(
\sumu{i\in I} b_{i,j} (\ecF_i+\ecD_i)+\eps_j\,\ecG_j
)\rceil\right).
\end{multline}
\item
On suppose qu'on a 
\begin{equation}
\acc{y}{\scG_2+\scG_{3}-\Dtot}
\geq 
\deg(\ecG_2)+\deg(\ecG_{3})
+2\,\gc-1
\end{equation}
\begin{equation}
\text{et}\quad\acc{y}{\scG_1+\frac{1}{2}\,\scG_{2}+\frac{1}{2}\,\scG_3-\Dtot}
\geq 
\deg(\ecG_1)+\frac{1}{2}\,\deg(\ecG_{2})+\frac{1}{2}\,\deg(\ecG_3)
+2\,\gc-1.
\end{equation} 
Alors la quantit\'e $\log_q \fonc_{\{1,2,3\}}(\becD,\becE)$ vaut
\begin{multline}
2\,(1-\gc)+\acc{y}{\sum_{j\in \{1,2,3\}}\scG_j-\Dtot}-\sum_{j\in \{1,2,3\}} \deg(\ecG_j)
+
\deg(\Infu{2\leq j\leq 3}(
\sumu{i\in I} b_{i,j} (\ecF_i+\ecD_i)+\ecG_j
))
\\
-
\deg\left(\lceil \frac{1}{2}
\Infu{2\leq j\leq 3}(
\sumu{i\in I} b_{i,j} (\ecF_i+\ecD_i)+\ecG_j
)-\frac{1}{2}
\Infu{1\leq j\leq 3}(
\sumu{i\in I} b_{i,j} (\ecF_i+\ecD_i)+\eps_j\,\ecG_j
)\rceil\right).
\end{multline}
\end{enumerate}
\end{prop}
Compte tenu de cette proposition, des hypoth\`eses 
suffisantes pour mener \`a bien la m\'ethode s'\'enoncent \`a pr\'esent ainsi.
\begin{hyps}\label{hyp:synth:dP6A2}
\begin{enumerate}
\item\label{item:hyp:synth:dP6A2:0}
On a
\begin{equation}
\scG_{2}+\scG_{3}-\Dtot\in \ceff(X)\setminus \{0\}
\end{equation}
\begin{equation}\label{eq:cond:synth:dP6A2}
\text{et}\quad \scG_{1}+\frac{1}{2}\,\scG_{2}+\frac{1}{2}\,\scG_3-\Dtot\in \ceff(X)\setminus \{0\}.
\end{equation}
\item\label{item:hyp:synth:dP6A2:i}
Pour $\{i,j,k\}=\{1,2,3\}$, la s\'erie
$Z(\frac{1}{2}\,\scG_i+\frac{1}{2}\,\scG_j,(0,\frac{1}{2},\frac{1}{2})$ (\cf
l'\'enonc\'e du lemme \ref{lm:muxconv}) est
$q^{-1}$-contr\^ol\'ee \`a l'ordre $\rg(\Pic(X))-1$.
\item\label{item:hyp:synth:dP6A2:ii}
La s\'erie $Z(\scG_1,(1,0,0))$
est
$q^{-1}$-contr\^ol\'ee \`a l'ordre $\rg(\Pic(X))-1$.
\item\label{item:hyp:synth:dP6A2:0:0}
L'analogue des hypoth\`eses \ref{hyp:conv:cEeps} et
\ref{hyp:rel:term:princ} pour la nouvelle forme 
\eqref{eq:def:Frhoalphabeta:dP6A2} de $F_{\rho,\be}$
est v\'erifi\'e.
\end{enumerate}
\end{hyps}
Bien entendu, le point \ref{item:hyp:synth:dP6A2:i} sert \`a contr\^oler
les termes d'erreur issus de la majoration \eqref{eq:maj:dP6A2:N123:bis} et le point
\ref{item:hyp:synth:dP6A2:ii} ceux issus de la majoration  \eqref{eq:maj:fonc:ijk:dP6A2}.
\begin{rem}\label{rem:hyp:synth:dP6A2}
On suppose que l'intersection $\scG_1\cap \scG_2\cap \scG_3$ est non
vide. On d\'eduit alors de la remarque \eqref{rem:inter:bis} les faits
suivants~:
\begin{enumerate}
\item
on suppose que pour $\{i,j,k\}=\{1,2,3\}$
il existe $\eps>0$ 
tel qu'on ait
\begin{equation}
(1-\eps)\,\scG_i+\frac{1}{2}\,\scG_j+\frac{1}{2}\,\scG_k
-\Dtot\in \ceff(X)\setminus \{0\}\quad;
\end{equation}
alors le point \ref{item:hyp:synth:dP6A2:i}
 des hypoth\`eses \ref{hyp:synth:dP6A2} est v\'erifi\'e~;
\item
on suppose qu'il existe $\eps>0$ tel qu'on ait
\begin{equation}
(1-\eps)\,\scG_2+(1-\eps)\,\scG_3-\Dtot\in \ceff(X)\setminus\{0\}\quad;
\end{equation}
alors le point \ref{item:hyp:synth:dP6A2:ii}
 des hypoth\`eses \ref{hyp:synth:dP6A2} est v\'erifi\'e.
\end{enumerate}
\end{rem}
\begin{rem}
Il est possible, \`a tr\`es peu de frais, d'am\'eliorer tr\`es l\'eg\`erement la proposition
\ref{prop:compt:bis}, ce qui conduit \`a affaiblir un peu les
hypoth\`eses ci-dessus. Une remarque similaire vaut pour la proposition
\ref{prop:compt}. Nous ne donnons pas ces \'enonc\'es raffin\'es d'une part pour ne
pas alourdir l'\'ecriture, d'autre part car ils sont inutiles pour les
applications propos\'ees.
Par exemple, la condition
\eqref{hyp:cor:clef:part:2:dP6A2} peut-\^etre affaiblie en demandant seulement l'existence
d'un $\alpha\in [0,1]$ v\'erifiant 
\begin{equation}
\deg(\ecH'_1)+\alpha\,\deg(\ecH'_2)+(1-\alpha)\,\deg(\ecH'_3)-\deg(\ecH)\geq 2\,\gc-1.
\end{equation}
ce qui permet de remplacer la condition \eqref{eq:cond:synth:dP6A2}
par la condition affaiblie
\beq
\exists\alpha\in [0,1],\quad \scG_{1}+\alpha\,\scG_{2}+(1-\alpha)\,\scG_3-\Dtot\in \ceff(X)\setminus \{0\}.
\eeq
\end{rem}

\subsection{V\'erification des hypoth\`eses dans le cas o\`u $X$ est la
  d\'esingularisation de la surface de del Pezzo singuli\`ere de degr\'e $6$
  avec singularit\'e $A_2$}
D'apr\`es \cite{Der:sdp:ut:hyp}, on se trouve dans le cas de figure
consid\'er\'e au d\'ebut de la section \ref{sec:app:sdp}. Plus pr\'ecis\'ement,
on a alors un isomorphisme
\begin{equation}
\Cox(X)\longisom k[(s_i)_{0\leq i \leq 3},(t_j)_{1\leq j \leq 3}]/
s_1\,t_1^2+s_2\,t_2+s_3\,t_3
\end{equation}
et les relations
\begin{align}
\scG_1&=\scF_0+\scF_2+\scF_3,\\
\scG_2&=2\,\scF_0+\scF_1+\scF_2+2\,\scF_3,\\
\text{et}\quad\quad\scG_3&=2\,\scF_0+\scF_1+2\,\scF_2+\scF_3.
\end{align}
Par ailleurs \cite{Der:sdp:ut:hyp} d\'ecrit enti\`erement les relations
d'incidence des diviseurs $\scF_i$ et $\scG_j$. En particulier on sait
que l'intersection $\scG_1\cap\scG_2\cap\scG_3$ est non vide.

La remarque \ref{rem:hyp:synth:dP6A2} et un calcul imm\'ediat montrent
alors que les points 
\ref{item:hyp:synth:dP6A2:0},
\ref{item:hyp:synth:dP6A2:i}
et \ref{item:hyp:synth:dP6A2:ii} des hypoth\`eses \ref{hyp:synth:dP6A2} sont v\'erifi\'ees.

Montrons \`a pr\'esent que les hypoth\`eses \ref{hyp:conv:cEeps} et
\ref{hyp:rel:term:princ} sont v\'erifi\'ees.
On a ici
\beq
F_{\rho,\be}(\bt)
=
G_{\rho,(f_1+2\,g_1,f_2+g_2,f_3+g_3)}(\bt)
\eeq
o\`u, pour $\bnu\in \bN^{\,3}$,
\begin{equation}
G_{\rho,\bnu}(\bt)
\eqdef
\sum_{\bd\in \bN^{\,\{0,1,2,3\}}}
\rho^{
\Min(d_2+\nu_2,d_3+\nu_3)
-\lceil
\frac{1}{2}(
\Min(d_2+\nu_2,d_3+\nu_3)
-
\Min(d_1+\nu_1,d_2+\nu_2,d_3+\nu_3)
)
\rceil
} \bt^{\,\bd}.
\end{equation}

Comme on l'a signal\'e dans la remarque \ref{rem:eff}, la v\'erification des hypoth\`eses \ref{hyp:conv:cEeps} et
\ref{hyp:rel:term:princ} peut \^etre confi\'ee \`a un logiciel de calcul formel.
Nous expliquons cependant ci-dessous comment l'hypoth\`ese
\ref{hyp:conv:cEeps} peut s'\'etablir via un minimum de calcul. En ce
qui concerne l'hypoth\`ese \ref{hyp:rel:term:princ}, nous n'avons pas
trouv\'e d'autre approche que le calcul par force brute. 

\begin{prop}\label{prop:calc:sergen:dP6A2}
Posons
\begin{equation}
\wt{G}_{\rho,\bnu}(\bt)
\eqdef
(1-\rho\,t_2^{\,2}\,t_3^{\,2})\,(1-\rho\,t_1\,t_2\,t_3)
\prod_{0\leq i\leq 3} (1-t_i)\,G_{\rho,\bnu}(\bt)\eqdef\sum_{\bd\in \bN^{\{1,2,3\}}}a_{\bnu,\bd}(\rho)\,\bt^{\bd}.
\end{equation}

Si {\em l'une} des conditions
\begin{align}
d_1&\geq 5+\abs{\nu_2-\nu_1}+\abs{\nu_3-\nu_1}\\
d_2&\geq 7+\abs{\nu_1-\nu_2}+\abs{\nu_3-\nu_2}\\
d_3&\geq 7+\abs{\nu_1-\nu_3}+\abs{\nu_2-\nu_3}
\end{align}
est satisfaite, $a_{\bnu,\bd}(\rho)$ est nul. En particulier, $\wt{G}_{\nu,\rho}(\bt)$ est un polyn\^ome.
\end{prop}
\begin{proof}
Pour $\bd\in \bN^3$, $\bmu\in \{0,1\}^3$ et $\bgamma\in \{0,1\}^2$
notons $\scP(\bd,\bmu,\bgamma)$ la condition
\beq
\left\{
\begin{array}{rrrr}
\mu_1+&\gamma_1\phantom{+}&&\leq d_1\\
\mu_2+&\gamma_1+&2\,\gamma_2&\leq d_2\\
\mu_3+&\gamma_1+&2\,\gamma_2&\leq d_3
\end{array}
\right.
\eeq
et
\begin{multline}\label{eq:def:varphi}
\varphi(\bd,\bnu,\bmu,\gamma_2)
\eqdef
\Min(d_2+\nu_2-\mu_2,d_3+\nu_3-\mu_3)
\\
-\lceil
\frac{1}{2}(
\Min(d_2+\nu_2-\mu_2,d_3+\nu_3-\mu_3)
-
\Min(d_1+\nu_1-\mu_1,d_2+\nu_2-\mu_2-2\,\gamma_2,d_3+\nu_3-\mu_3-2\,\gamma_2)
)
\rceil
\end{multline}
Un calcul \'el\'ementaire montre
que le coefficient $a_{\bnu,\bd}(\rho)$ vaut
\begin{equation}\label{eq:coef:anubd}
\sum_{\substack{
\bgamma\in \{0,1\}^2,\,\bmu\in \{0,1\}^3\\
\scP(\bd,\bmu,\bgamma)
}}
(-1)^{\sum \gamma_i+\sum \mu_i}
\rho^{\varphi(\bd,\bnu,\bmu,\gamma_2)}.
\end{equation}
Supposons qu'on ait $d_1\geq 5+\abs{\nu_2-\nu_1}+\abs{\nu_1-\nu_3}$. 
En particulier la condition $\scP(\bd,\bmu,\bgamma)$ ne d\'epend plus de $\mu_1$.
Si on a en outre $d_2\geq 4$ et $d_3\geq 4$, la condition
$\scP(\bd,\bmu,\bgamma)$ ne d\'epend plus de $\gamma_1$, et comme
$\varphi$ ne d\'epend pas non plus de $\gamma_1$,
l'expression \eqref{eq:coef:anubd} montre aussit\^ot que $a_{\nu,\bd}$ est nul. On
peut donc supposer $d_2\leq 3$. Mais alors l'hypoth\`ese sur $d_1$ entra\^\i ne aussit\^ot qu'on a
pour tout $(\bmu,\bgamma)\in \{0,1\}^5$
\beq
d_1+\nu_1-\mu_1\geq d_2+\nu_2-\mu_2-2\,\gamma_2.
\eeq
Ainsi $\varphi(\bd,\bnu,\bmu,\gamma_2)$ ne d\'epend pas de $\mu_1$, et
l'expression \eqref{eq:coef:anubd} montre aussit\^ot que $a_{\bnu,\bd}$ est nul.
Les autres cas se traitent de mani\`ere similaire.
\end{proof}
D'apr\`es \eqref{eq:coef:anubd}
et \eqref{eq:def:varphi},
le degr\'e en $\rho$ de $a_{\bnu,\bd}(\rho)$ est major\'e
par $\Min(\nu_2+d_2,\nu_3+d_3)$. 
Or, si $(d_1,d_2,d_3)$ est non nul, on v\'erifie aussit\^ot que la majoration
\beq
\Min(d_2,d_3)<d_1+d_2+d_3-1 
\eeq
vaut sauf si on a 
\beq
(d_1,d_2,d_3)\in \{(0,0,1),(0,1,0),(0,1,1),(1,0,0)\}
\eeq
et dans ce cas l'expression \eqref{eq:coef:anubd}  montre 
facilement que $a_{0,\bd}$ est nul.
Ainsi l'hypoth\`ese \ref{hyp:conv:sergen} est v\'erifi\'ee.

Par ailleurs le lemme \ref{lm:majem} montre qu'on a pour tout $\bnu$
et tout $\bd$ la majoration
\beq\label{eq:majdeg}
\deg(a_{\bnu,\bd}(\rho))\leq d_1+d_2+d_3+\Min(\nu_2,\nu_3).
\eeq
Si $\bnu=(0,1,1)$, toujours d'apr\`es \eqref{eq:coef:anubd}
et \eqref{eq:def:varphi}, on a 
\beq
\deg(a_{(0,1,1),\bd}(\rho))\leq \frac{1}{2}+\frac{1}{2}\Min(d_2,d_3)+\frac{1}{2}\Min(d_1,d_2+1,d_3+1)
\eeq
et donc la majoration ci-dessus peut-\^etre am\'elior\'ee
en 
\beq\label{eq:majdegbis}
\deg(a_{(0,1,1),\bd}(\rho))\leq d_1+d_2+d_3+\frac{1}{2}.
\eeq

La remarque \ref{rem:hyp:conv:cEeps} et \eqref{eq:majdeg}
montrent qu'il suffit pour v\'erifier l'hypoth\`ese \ref{hyp:conv:cEeps} 
d'\'etablir que si $\be\in \{0,1\}^7$ est non nul et tel que $\mu_{X}^0(\be)$ est non nul on a 
\begin{equation}
-\sum_{1\leq j\leq 3} g_j-\sum_{0\leq i\leq 3} f_i+\min(f_2+g_2,f_3+g_3) < -1.
\end{equation}
On v\'erifie aussit\^ot que ceci vaut sauf dans le cas o\`u $(f_0,f_1,g_1)=(0,0,0)$
et $(f_2+g_2,f_3+g_3)=(1,1)$. Mais alors \eqref{eq:majdegbis}
montre qu'il suffit en fait d'avoir
\begin{equation}
-\sum_{1\leq j\leq 3} g_j-\sum_{0\leq i\leq 3} f_i+\frac{1}{2}\min(f_2+g_2,f_3+g_3) < -1.
\end{equation}

\bibliographystyle{amsalpha}

\begin{thebibliography}{Bou09}

\bibitem[BH04]{BerHau:bun}
Florian Berchtold and J{\"u}rgen Hausen, \emph{Bunches of cones in the divisor
  class group---a new combinatorial language for toric varieties}, Int. Math.
  Res. Not. (2004), no.~6, 261--302. \MR{MR2041065 (2004m:14111)}

\bibitem[BH07]{BerHau:Cox}
\bysame, \emph{Cox rings and combinatorics}, Trans. Amer. Math. Soc.
  \textbf{359} (2007), no.~3, 1205--1252 (electronic). \MR{MR2262848
  (2007h:14007)}

\bibitem[Bou09]{Bou:compt}
David Bourqui, \emph{Comptage de courbes sur le plan projectif \'eclat\'e en
  trois points align\'es}, Ann. Inst. Fourier (Grenoble) \textbf{59} (2009),
  no.~5, 1847--1895.

\bibitem[CT02]{CLT:inv}
Antoine Chambert-Loir and Yuri Tschinkel, \emph{On the distribution of
  points of bounded height on equivariant compactifications of vector
  groups}, Invent. Math.  \textbf{148} (2002),
  no.~2, 421--452.

\bibitem[Der06]{Der:sdp:ut:hyp}
Ulrich Derenthal, \emph{Singular {D}el {P}ezzo surfaces whose universal torsors
  are hypersurfaces}, {\url{arXiv:math/0604194v1}}, 2006.

\bibitem[Has04]{Has:eq:ut:cox:rings}
B.~Hassett, \emph{Equations of universal torsors and {C}ox rings},
  Mathematisches {I}nstitut, {G}eorg-{A}ugust-{U}niversit\"at {G}\"ottingen:
  {S}eminars {S}ummer {T}erm 2004, Universit\"atsdrucke G\"ottingen,
  G\"ottingen, 2004, pp.~135--143. \MR{MR2183138 (2007a:14046)}

\bibitem[HK00]{hukeel:mori}
Yi~Hu and Sean Keel, \emph{Mori dream spaces and {GIT}}, Michigan Math. J.
  \textbf{48} (2000), 331--348, Dedicated to William Fulton on the occasion of
  his 60th birthday. \MR{MR1786494 (2001i:14059)}

\bibitem[Pey95]{Pey:duke}
Emmanuel Peyre, \emph{Hauteurs et mesures de {T}amagawa sur les vari\'et\'es de
  {F}ano}, Duke Math. J. \textbf{79} (1995), no.~1, 101--218. \MR{MR1340296
  (96h:11062)}

\bibitem[Sal98]{Sal:tammes}
Per Salberger, \emph{Tamagawa measures on universal torsors and points of
  bounded height on {F}ano varieties}, Ast\'erisque (1998), no.~251, 91--258,
  Nombre et r\'epartition de points de hauteur born\'ee (Paris, 1996).
  \MR{MR1679841 (2000d:11091)}

\end{thebibliography}

\providecommand{\bysame}{\leavevmode\hbox to3em{\hrulefill}\thinspace}
\providecommand{\MR}{\relax\ifhmode\unskip\space\fi MR }
\providecommand{\MRhref}[2]{
  \href{http://www.ams.org/mathscinet-getitem?mr=#1}{#2}
}
\providecommand{\href}[2]{#2}

\end{document}